\RequirePackage{fix-cm}
\documentclass[smallextended]{svjour3}       % onecolumn (second format)
\smartqed  % flush right qed marks, e.g. at end of proof
\usepackage{graphicx}
\usepackage{amsmath, ntheorem}
\usepackage{amssymb}
\usepackage{stackrel}
\usepackage{marvosym, enumerate}
\usepackage{amstext, amscd, latexsym}
\usepackage{amsfonts, ragged2e}
\usepackage{mathrsfs, pgfplots, enumerate, setspace}
\usepackage{subfigure}
\usepackage{cases}
\usepackage{epstopdf}
\usepackage{booktabs}
\usepackage{tabularx}
\usepackage[ruled,linesnumbered]{algorithm2e}
\usepackage{float} %fix the table and figure
\usepackage[font={bf},textfont=md,labelfont={footnotesize},labelformat={default}]{caption}
\numberwithin{equation}{section}
\smartqed
\usepackage{cite}
\usepackage[colorlinks,linkcolor=blue,anchorcolor=blue,citecolor=blue]{hyperref}% ³¬Á´½ÓµÄÑÕÉ«
\usepackage{amstext, graphicx, amscd, latexsym, amssymb}
\usepackage{amsfonts, ragged2e}
\usepackage{pgfplots, setspace}
\usepackage{geometry}
\usepackage{latexsym,bm}
\usepackage{float}
\usepackage{threeparttable}
\usepackage{array}
\usepackage{booktabs}
\usepackage{makecell}
\usepackage[justification=centering,labelfont=bf]{caption}
\usepackage{etoolbox}
\usepackage[figuresright]{rotating}
\begin{document}

\title{A Branch and Bound Algorithm for Multiobjective Optimization Problems Using General Ordering Cones}

\titlerunning{}        % if too long for running head

\author{Weitian Wu$^1$ \and Xinmin Yang$^2$ }

%\authorrunning{} % if too long for running head

\institute{ W.T. Wu \at School of Statistics and Data Sciences, Ningbo University of Technology 315211, China\\
                    \email{weitianwu@nbut.edu.cn}\\
           \Letter X.M. Yang \at National Center for Applied Mathematics of Chongqing 401331, China\\
           School of Mathematical Sciences, Chongqing Normal University, Chongqing 401331, China\\
              xmyang@cqnu.edu.cn}

\date{Received: date / Accepted: date}

\maketitle

\begin{abstract}

%Many branch and bound algorithms for multiobjective optimization problems are considered computationally expensive to approximate the entire Pareto optimal solutions set. In this paper, we propose a new branch and bound algorithm to approximate solutions that are optimal with respect to general cones that contain the nonnegative orthant of the objective space.

Many existing branch and bound algorithms for multiobjective optimization problems require a significant computational cost to approximate the entire Pareto optimal solution set. In this paper, we propose a new branch and bound algorithm that approximates a part of the Pareto optimal solution set by introducing the additional preference information in the form of ordering cones. The basic idea is to replace the Pareto dominance induced by the nonnegative orthant with the cone dominance induced by a larger ordering cone in the discarding test. In particular, we consider both polyhedral and non-polyhedral cones, and propose the corresponding cone dominance-based discarding tests, respectively. In this way, the subboxes that do not contain efficient solutions with respect to the ordering cone will be removed, even though they may contain Pareto optimal solutions. We prove the global convergence of the proposed algorithm. Finally, the proposed algorithm is applied to a number of test instances as well as to 2- to 5-objective real-world constrained problems.

\end{abstract}

\section{Introduction}
Multiobjective optimization problems (MOPs) are of growing interest in both mathematical optimization theory and real-world applications. In these problems, multiple conflicting objectives must be considered simultaneously, making it impossible to find a single solution that is optimal for all objectives. Instead, a set of Pareto optimal (efficient) solutions can be identified, characterized by the fact that the improvement of one objective can only be achieved at the expense of the deterioration of at least one other objective. In the absence of prior preference information, none of the Pareto optimal solutions can be said to be inferior to others. Consequently, one of the major challenges in multiobjective optimization is to assist the decision maker in making trade-offs between multiple objectives and in identifying a Pareto optimal solution that is most satisfactory to he/she.

The a posteriori methods attempt to generate a set of well-distributed representative solutions of the entire Pareto optimal solution set. The decision maker then has to examine this potential set of alternative solutions and make a choice. The majority of existing branch and bound algorithms for MOPs, as referenced in \cite{ref10,ref11,ref26,ref42,ref43,ref55,ref56,ref57,ref68,ref69}, can be classified as a posteriori. However, the solution processes of these algorithms are considered to be resource-intensive and time-consuming. This is because high-precision and well-distributed solutions can only be obtained by continuously subdividing the variable space, and the number of subboxes may increase exponentially during the subdivision (branching) process. Furthermore, the final decision relies heavily on the image of the Pareto optimal solution set, commonly known as the Pareto front. In the case of problems with four or more objectives, not only is the visualization of the Pareto front not intuitive, but also the number of representative solutions can be huge. In this case, even if these branch and bound algorithms successfully approximate the complete Pareto optimal solution set, it is challenging for the decision maker to make a choice.

One method for reducing the computational cost and alleviating decision pressure is to incorporate additional preference information into the branch and bound process. For instance, Wu and Yang \cite{ref44} employed preference information expressed through reference points to direct the search towards the corresponding regions of interest, thus circumventing the exhaustive search for the entire Pareto optimal solution set. Eichfelder and Stein \cite{ref60} introduced a bounded trade-off to truncate all subboxes in the solution process, thus greatly reducing the number of subboxes to be explored. It should be noted that in multiobjective optimization, a common way to model preferences is through ordering cones. For example, in \cite{ref61}, convex polyhedral ordering cones are used to model a decision maker's preferences based on trade-off information. In portfolio optimization \cite{ref62}, the dominance relationship between two portfolios is defined by the so-called ice cream cone, which is non-polyhedral. Therefore, ordering cones are essential for decision-making in many applications.

A well-known property of ordering cones is that the larger the ordering cone, the smaller the solution set obtained. In the context of branch and bound algorithms, which are a posteriori methods, the ordering cone corresponds to the nonnegative orthant of the objective space. Consequently, a portion of the Pareto optimal set can be approximated by using an ordering cone that is larger than the nonnegative orthant. In this paper, we propose a cone dominance-based branch and bound algorithm for MOPs. This algorithm is designed to approximate a portion of the Pareto optimal set rather than the entire set. The discarding test of the proposed algorithm employs a cone dominance relation induced by a general ordering cone, rather than the Pareto dominance relation, where the general ordering cone contains the nonnegative orthant. Therefore, the new discarding test is capable of excluding the subboxes which do not contain efficient solutions with respect to the corresponding cone. In particular, we construct a polyhedral cone to identify the $\epsilon$-properly Pareto optimal solutions \cite{ref38} whose trade-offs are bounded by $\epsilon$ and $1/\epsilon$. Furthermore, we consider non-polyhedral cones in three-dimensional or higher-dimensional objective space, and further discuss their properties and computation. We demonstrate the global convergence of the proposed algorithm. Finally, the algorithm is applied to several benchmark problems as well as to two- to five-objective real-world constrained optimization problems.

The rest of this paper is organized as follows. In Section 2, we introduce the basic concepts and notations which will be used in the sequel. The \emph{cone dominance-based branch and bound algorithm} and its theoretical analysis is described in Section 3. Section 4 is devoted to some numerical results.

\section{Preliminaries}
In this section we introduce the basic concepts which we need for the new algorithm. Let us have a pointed closed convex cone $\mathcal{C}$ defined in $\mathbb{R}^m$. The partial order $\leqq_\mathcal{C}$ induced by the cone $C$ is given by $y^1\leqq_\mathcal{C} y^2$ if and only if $y^2-y^1\in \mathcal{C}$, where $y^1$ and $y^2$ are two point in $\mathbb{R}^m$. Furthermore, $y^1\leq_\mathcal{C} y^2$ is written if and only if $y^2-y^1\in \mathcal{C}\backslash \{0\}$, and $y^1 <_\mathcal{C} y^2$ is written if and only if $y^2-y^1\in {\rm int}\mathcal{C}$, where ${\rm int}$ denotes the interior. For $\mathcal{C}=\mathbb{R}^m_+$, the abbreviated notation $y^1\leqq y^2$ is used if and only if $y^2-y^1\in \mathbb{R}^m_+$, and $y^1\leq y^2$ if and only if $y^2-y^1\in \mathbb{R}^m_+\backslash\{0\}$.

We study the following optimization problem:
\begin{align}\label{VP}
\mathcal{C}\hbox{-}\min\limits_{x\in\Omega}\quad F(x)=(f_1(x),\ldots ,f_m(x))^T
\end{align}
with
\begin{align*}
\Omega=\{x\in\mathbb{R}^n:g_j(x)\geq0,\;j=0,\ldots,p,\;\underline{x}_k\leq x_k\leq \overline{x}_k,\;k=0,\ldots,n\},
\end{align*}
where $f_i:\mathbb{R}^n\rightarrow \mathbb{R}$ ($i=1,\ldots,m$) are Lipschitz continuous, and $g_j:\mathbb{R}^n\rightarrow \mathbb{R}$ ($j=0,\ldots,p$) are continuous. If we allow $j=0$, the set $\Omega$ is referred to as a box constraint. In this case, we call $\Omega$ a \emph{box} with the midpoint $m(\Omega)=(\frac{\underline{x}_1+\overline{x}_1}{2},\ldots,\frac{\underline{x}_n+\overline{x}_n}{2})^T$ and the width $\omega(\Omega)=(\overline{x}_1-\underline{x}_1,\ldots,\overline{x}_n-\underline{x}_n)^T$. The diameter of $\Omega$ is denoted by $\|\omega(\Omega)\|$. For a feasible solution $x\in\Omega$, the objective vector $F(x)\in\mathbb{R}^m$ is said to be the image of $x$, while $x$ is called the preimage of $F(x)$. The notation ``$\mathcal{C}\hbox{-}\min$'' means that finding the minimum with respect to the cone $\mathcal{C}$.

The concept of cone dominance relation between two solutions $x^1,x^2\in\Omega$ can be defined as follows:
\begin{align*}
x^1\;\mathcal{C}\hbox{-}dominates\;x^2~\Longleftrightarrow\;F(x^1)\leq_\mathcal{C} F(x^2)\;\Longleftrightarrow\;F(x^2)-F(x^1)\in \mathcal{C}\backslash\{0\}.
\end{align*}

A point $x^*\in\Omega$ is \emph{efficient} of problem \eqref{VP} with respect to the cone $\mathcal{C}$ if there does not exist another $x\in\Omega$ such that $F(x)\leq_\mathcal{C} F(x^*)$. The set of all efficient solutions with respect to the cone $\mathcal{C}$ is denoted as $\mathcal{M}(F(\Omega),\mathcal{C})$. A nonempty set $\mathcal{U}(\mathcal{C})\subseteq\mathbb{R}^m$ is called a \emph{nondominated set with respect to $\mathcal{C}$} if for any $y^1,y^2\in \mathcal{U}(\mathcal{C})$ we have $y^1\nleqq_\mathcal{C} y^2$ and $y^2\nleqq_\mathcal{C} y^1$.

For $\mathcal{C}=\mathbb{R}^m_+$, the cone dominance is equivalent to Pareto dominance:
\begin{align*}
x^1\;dominates\;x^2~\Longleftrightarrow\;F(x^1)\leq F(x^2)\;\Longleftrightarrow\;F(x^2)-F(x^1)\in \mathbb{R}^m_+\backslash\{0\},
\end{align*}
and the efficient solutions are also known as the \emph{Pareto optimal solutions}. The set of all Pareto optimal solutions is called the \emph{Pareto optimal set}. The image of Pareto optimal set under the mapping $F$ is called the \emph{Pareto front}.

The aim of an approximation algorithm is used to found an $\varepsilon$-efficient solution, which is defined next. %Let $e$ denote the $m$-dimensional all-ones vector $(1,\ldots,1)^T\in\mathbb{R}^m$.

\begin{definition}\cite{ref63}
For given $\varepsilon\in\mathcal{C}\backslash\{0\}$. A point $\bar{x}\in \Omega$ is an $\varepsilon$-efficient solution with respect to the cone $\mathcal{C}$ if there does not exist another $x\in \Omega$ with $F(x)\leqq_{\mathcal{C}} F(\bar{x})-\varepsilon$.
\end{definition}

We use $d(a,b)=\|a-b\|$ to quantify the distance between two points $a$ and $b$, where $\|\cdot\|$ denotes the Euclidean norm. The distance between the point $a$ and a nonempty finite set $B$ is defined as $d(a,B):=\min_{b\in B}\|a-b\|.$ Let $A$ be another non-empty finite set, we define the Hausdorff distance between $A$ and $B$ by
\begin{align*}
d_H(A,B):=\max\{d_h(A,B),d_h(B,A)\},
\end{align*}
where $d_h(A,B)$ is the directed Hausdorff distance from $A$ to $B$, defined by
\begin{align*}
d_h(A,B):=\max_{a\in A}\{\min_{b\in B}\|a-b\|\}.
\end{align*}

Branch and bound algorithms \cite{ref10,ref11,ref26,ref42,ref43,ref55,ref56,ref57} have been employed to solve multiobjective optimization problems. By means of a tree search, a branch and bound algorithm systematically searches for an approximation of the entire Pareto optimal set. The basic branch and bound algorithm for MOPs was initially proposed by Fern{\'a}ndez and T{\'o}th \cite{ref11} (see Algorithm \ref{alg1}). The solution process is comprised of three components:
\begin{itemize}
  \item \emph{branching}: subboxes are bisected perpendicularly to the direction of maximum width;
  \item \emph{bounding}: the lower and upper bounds for subboxes are calculated;
  \item \emph{pruning}: the subboxes that are provably suboptimal are excluded from exploration.
\end{itemize}

\begin{algorithm}[H]\label{alg1}
  %\SetNoFillComment
  \SetKwInOut{Input}{Input}\SetKwInOut{Output}{Output}
  \Input{an MOP, termination criterion;}
  \Output{$\mathcal{B}_{k}$, $\mathcal{U}(\mathbb{R}^m_+)$;}
  \BlankLine
  $\mathcal{B}_0\leftarrow \Omega$, $\mathcal{U}^{nds}\leftarrow \emptyset$, $k=0$\;
  \While{termination criterion is not satisfied}
  {$\mathcal{B}_{k+1}\leftarrow \emptyset$\;
  \While{$\mathcal{B}_k\neq\emptyset$}{
  Select $B\in\mathcal{B}_k$ and remove it from $\mathcal{B}_k$\;
  $B_1,B_2\longleftarrow$ Bisect $B$ perpendicularly to the direction of maximum width\;
  \For{$i=1,2$}
  {Calculate the lower bound $l(B_i)$ and upper bound $u(B_i)$ for $B_i$\;
   \If{$B_i$ can not be discarded}{
        Update $\mathcal{U}(\mathbb{R}^m_+)$ by $u(B_i)$ and store $B_i$ into $\mathcal{B}_{k+1}$\;}}
  }
  $k\leftarrow k+1$.}
  \caption{A basic branch and bound algorithm}
\end{algorithm}
\bigskip

The upper bound of a subbox $B$ may be defined as the image of any feasible point in $B$. In practice, the midpoint or the vertices of $B$ are typically selected. The approaches for the lower bounds proposed so far in the literature include the natural interval extension \cite{ref11,ref25}, the Lipschitz bound \cite{ref42,ref43} and the $\alpha$BB method \cite{ref26}, and the resulting lower bound $l=(l_1,\ldots,l_m)^T$ of $B$ satisfies
\begin{align}
  l\leq F(x),\quad  x\in B.\label{IE:2.2}
\end{align}
Numerical experiments indicate that there is no significant difference between the three bounding approaches. However, the latter two calculate the maximal error between the lower bounds and optimal values.

The pruning can be achieved by \emph{discarding tests}. A discarding test is capable of limiting the tree search, thereby avoiding exhaustive enumeration. A common type of discarding test is based on the Pareto dominance:

\emph{A subbox will be discarded if there exists a feasible objective vector such that the objective vector dominates the lower bound of the subbox.}

It is evident that the subbox is removed because it does not contain any Pareto optimal solutions.

\section{Cone dominance-based branch and bound algorithm}
In this section, we propose the cone dominance-based branch and bound algorithm. As mentioned earlier, we will consider general ordering cones that are larger than the nonnegative orthant in the objective space. Thus, we have the following lemma about the resulting partial orderings.
\begin{lemma}\label{le:1}
Let $\mathcal{C}$ be a pointed closed convex cone in $\mathbb{R}^m$ and satisfy $\mathcal{C}\supseteq \mathbb{R}^m_+$. Let $\leqq_\mathcal{C}$ and $\leq_\mathcal{C}$ be the partial orderings characterized by $\mathcal{C}$. For $y^1,y^2\in\mathbb{R}^m$, if $y^1\leqq(\leq) y^2$, we have $y^1\leqq_\mathcal{C}(\leq_\mathcal{C})y^2$.
\end{lemma}
\begin{proof}
This conclusion is derived from the definition of Pareto dominance. \qed
\end{proof}
%Next we will consider the case of non-polyhedral cones.
\subsection{Polyhedral ordering cone}
The polyhedral ordering cones are defined as follows:
\begin{definition}\label{de:2}\cite{ref63}
A set $\mathcal{C}\subseteq\mathbb{R}^m$ is a polyhedral cone if there exists a matrix $M\in\mathbb{R}^{s\times m}$ such that
$\mathcal{C}=\{y\in\mathbb{R}^m:My\geqq 0\}$. The kernel of a polyhedral cone is defined as the kernel (or nullspace) of the associated matrix, $\hbox{Ker}\,\mathcal{C}:=\hbox{Ker}\,M=\{y\in\mathbb{R}^m:My=0\}$.
\end{definition}

For $0\leq\epsilon<1$, now we define a linear mapping $\mathcal{T}_{\epsilon}:\mathbb{R}^m\rightarrow\mathbb{R}^m$,
\begin{align*}
  \mathcal{T}_{\epsilon}(y):=\begin{bmatrix}
1\quad & \quad \epsilon \quad &\quad \cdots\quad&\quad \epsilon\\
\epsilon\quad&\quad1\quad&\quad\cdots\quad&\quad\epsilon\\
\vdots\quad & \quad\vdots\quad &\quad \ddots\quad &\quad \vdots\\
\epsilon\quad&\quad \epsilon\quad &\quad \cdots\quad &\quad 1
\end{bmatrix}\cdot y.
\end{align*}
Using this notation, we define a set
\begin{align*}
  \mathcal{C}_{\epsilon}:=\{y\in\mathbb{R}^m:\mathcal{T}_{\epsilon}(y)\geqq 0\}.
\end{align*}
By Definition \ref{de:2}, the set $\mathcal{C}_{\epsilon}$ is a polyhedral ordering cone, and thus the corresponding cone orderings $\leqq_{\mathcal{C}_{\epsilon}}$ and $\leq_{\mathcal{C}_{\epsilon}}$ in $\mathbb{R}^m$ can be characterized. The following lemmas hold for $\leqq_{\mathcal{C}_{\epsilon}}$ and $\leq_{\mathcal{C}_{\epsilon}}$.
%\begin{align*}
%&y^1\leqq_{\mathcal{C}_{\epsilon}}y^2\Longleftrightarrow y^2-y^1\in\mathcal{C}_{\epsilon};\\
%&y^1\leq_{\mathcal{C}_{\epsilon}}y^2\Longleftrightarrow y^2-y^1\in\mathcal{C}_{\epsilon}\backslash\{0\}.
%\end{align*}
%In particular for $\epsilon=0$, the polyhedral cone $\mathcal{C}_{\epsilon}$ coincides with $\mathbb{R}^m_+$.
\begin{lemma}\label{le:2}
For $y^1,y^2\in\mathbb{R}^m$, we have
\begin{align*}
y^1\leqq_{\mathcal{C}_{\epsilon}}(\leq_{\mathcal{C}_{\epsilon}})y^2\Leftrightarrow \mathcal{T}_{\epsilon}(y^1)\leqq(\leq)\mathcal{T}_{\epsilon}(y^2)
%y^1\leq_{\mathcal{C}_{\epsilon}}y^2\Leftrightarrow \mathcal{T}_{\epsilon}(y^1)\leq\mathcal{T}_{\epsilon}(y^2)
\end{align*}
\end{lemma}
\begin{proof}
By the definition of $\leqq_{\mathcal{C}_{\epsilon}}$ and $\mathcal{C}_{\epsilon}$ and by the linearity of $\mathcal{T}_{\epsilon}$, we have
\begin{align*}
  y^1\leqq_{\mathcal{C}_{\epsilon}}y^2\Leftrightarrow y^2-y^1\in\mathcal{C}_{\epsilon}\Leftrightarrow \mathcal{T}_{\epsilon}(y^2-y^1)\geqq 0\Leftrightarrow \mathcal{T}_{\epsilon}(y^2)\geqq \mathcal{T}_{\epsilon}(y^1).
\end{align*}
Furthermore, it is easy to see that Ker~$\mathcal{C}_{\epsilon}=\{0\}$, it follows that $y^1\leq_{\mathcal{C}_{\epsilon}}y^2\Leftrightarrow \mathcal{T}_{\epsilon}(y^2)\geq \mathcal{T}_{\epsilon}(y^1).$ \qed
\end{proof}

\begin{lemma}\label{le:3}
The cone ordering $\leq_{\mathcal{C}_{\epsilon}}$ is a strict partial ordering.
\end{lemma}
\begin{proof}
By Definition 2.3.1 in \cite{ref58}, we would like to show that $\leq_{\mathcal{C}_{\epsilon}}$ is irreflexive and transitive. First, it is easy to see that for all $y\in \mathbb{R}^m$, $y\nleq_{\mathcal{C}_{\epsilon}} y$, i.e., $\leq_{\mathcal{C}_{\epsilon}}$ is irreflexive. Next we will prove $\leq_{\mathcal{C}_{\epsilon}}$ is transitive. For $y^1,y^2,y^3\in\mathbb{R}^m$, assume $y^1\leq_{\mathcal{C}_{\epsilon}} y^2$ and $y^2\leq_{\mathcal{C}_{\epsilon}} y^3$. By Lemma \ref{le:2}, we have
$\mathcal{T}_{\epsilon}(y^1)\leq\mathcal{T}_{\epsilon}(y^2)$ and $\mathcal{T}_{\epsilon}(y^2)\leq\mathcal{T}_{\epsilon}(y^3)$. Due to the transitivity of $\leq$, we have $\mathcal{T}_{\epsilon}(y^1)\leq\mathcal{T}_{\epsilon}(y^3)$, meaning that $y^1\leq_{\mathcal{C}_{\epsilon}} y^3$. \qed
\end{proof}

%\begin{lemma}{\rm \cite{ref22}}\label{le:3}
%For $y^1,y^2\in\mathbb{R}^m$, if $y^1\leq y^2$, we have $y^1\leq_{\mathcal{C}_{\epsilon}}y^2$.
%\end{lemma}
%\begin{proof}
%  According to definitions of the Pareto dominance and $\mathcal{C}_{\epsilon}$, we know $y^2-y^1\in\mathbb{R}^m_+\backslash\{0\}\subseteq\mathcal{C}_{\epsilon}\backslash\{0\}$, following $\mathcal{T}_{\epsilon}(y^2-y^1)\geq0$. By the linearity of $\mathcal{T}_{\epsilon}$ and Lemma \ref{le:1}, we have $y^1\leq_{\mathcal{C}_{\epsilon}}y^2$. \qed
%\end{proof}

According to above lemmas, the $\mathcal{C}_{\epsilon}$-dominance relation between $x^1,x^2\in\Omega$ can be defined as follows:
\begin{align*}
x^1~\mathcal{C}_{\epsilon}\hbox{-}dominates~x^2\Longleftrightarrow F(x^1)\leq_{\mathcal{C}_{\epsilon}} F(x^2)\Longleftrightarrow\mathcal{T}_{\epsilon}(F(x^1))\leq\mathcal{T}_{\epsilon}(F(x^2)).
\end{align*}

Here we consider the $\epsilon$-properly Pareto optimal solution proposed by Wierzbicki \cite{ref38}:

\begin{definition}\label{de:1}
The solution $x^*\in\Omega$ is said to be the $\epsilon$-properly Pareto optimal solution of problem \eqref{VP}, if
\begin{align*}
(F(x^*)-\mathbb{R}^m_{\epsilon}\backslash\{0\})\cap F(\Omega)=\emptyset,
\end{align*}
where $\mathbb{R}^m_{\epsilon}=\{y\in\mathbb{R}^m:\min_{i=1,\ldots,m} (1-\epsilon)y_i+\epsilon\sum_{i=1}^{m}y_i\geq 0\}$, $0\leq\epsilon<1$.
\end{definition}

Figure \ref{fig2} depicts the $\epsilon$-properly Pareto optimal solution of the bi-objective optimization problem. The $\epsilon$-properly Pareto optimal solution can be obtained by intersecting the feasible region with a blunt cone $\mathbb{R}^m_{\epsilon}$. Compared to the Pareto optimal solution, the $\epsilon$-properly Pareto optimal solution uses a larger set $\mathbb{R}^m_{\epsilon}$ instead of $\mathbb{R}^m_+$, so the $\epsilon$-properly Pareto optimal solution set is contained in the Pareto optimal solution set. Furthermore, an interesting aspect of $\epsilon$-properly Pareto optimal solutions is that the trade-offs are bounded by $\epsilon$ and $1/\epsilon$ \cite{ref23,ref59}.

%Additionally, the most important thing is that the trade-off rates of $\epsilon$-properly Pareto optimal solutions are bounded by $\epsilon$ and $1/\epsilon$, which inspires us to associate the knees with $\epsilon$-properly Pareto optimal solutions.

\begin{figure}[h]%
\centering
\includegraphics[width=0.5\textwidth]{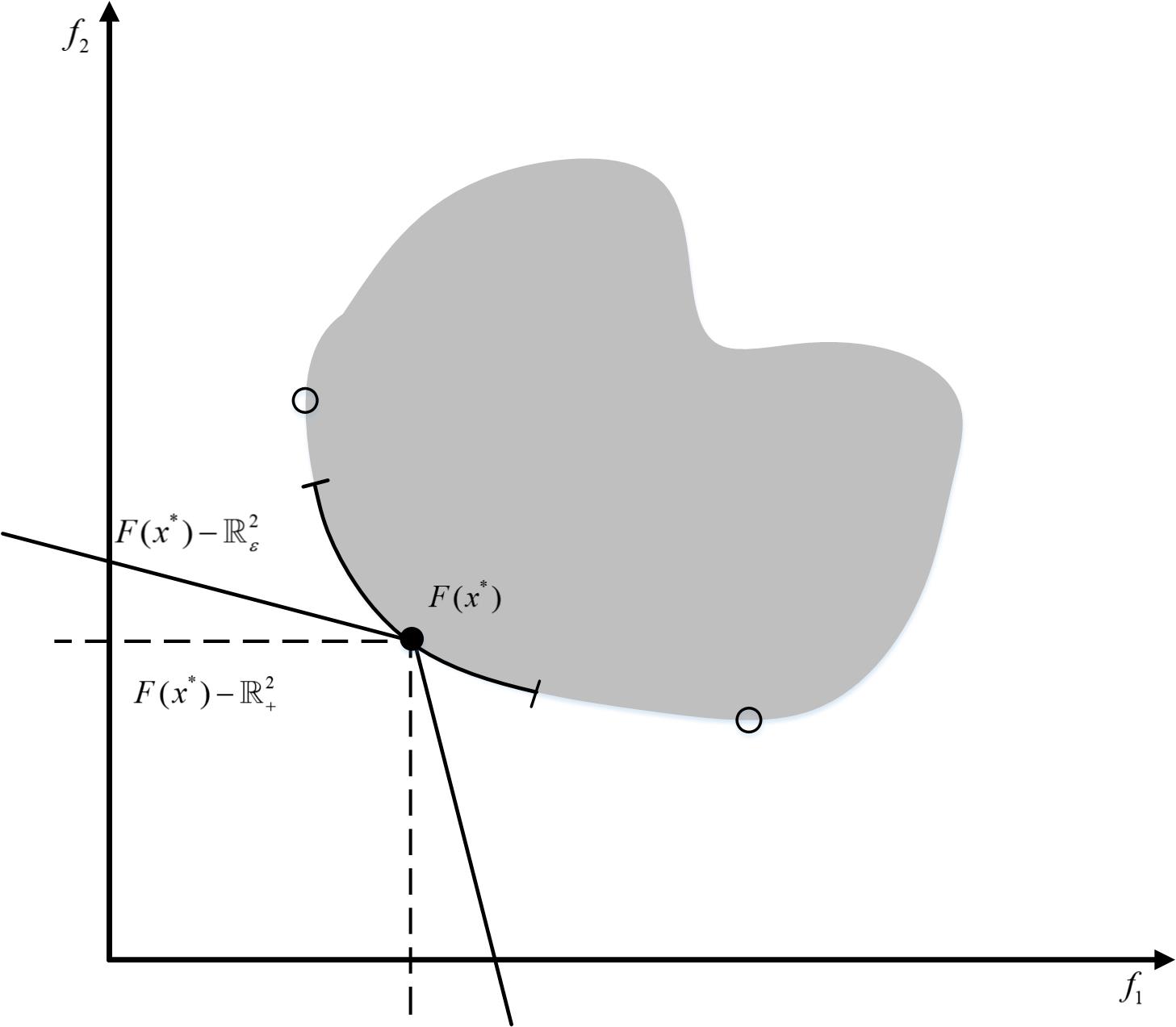}
\caption{$\epsilon$-properly Pareto optimal solution}\label{fig2}
\end{figure}

The following theorem discuss the relationship between $\mathcal{C}_{\epsilon}$-dominance and $\epsilon$-properly Pareto optimal solutions.

\begin{theorem}\label{th:1}
   For $\epsilon\in[0,1)$, we have $x$ is an $\epsilon$-properly Pareto optimal solution if and only if there does not exist $x'\in\Omega$, such that $\mathcal{T}_{\epsilon}(F(x'))\leq \mathcal{T}_{\epsilon}(F(x))$.
\end{theorem}
\begin{proof}
    First of all, It is easy to see that $\mathbb{R}^m_{\epsilon}=\mathcal{C}_{\epsilon}$.

    ($\Rightarrow$) Assume that $x$ is an $\epsilon$-properly Pareto optimal solution. Suppose, to the contrary, that there exists $x'\in \Omega$, such that $F(x')\leq_{\mathcal{C}_{\epsilon}} F(x)$. By Lemma \ref{le:2}, we have
\begin{align*}
  F(x')\leq_{\mathcal{C}_{\epsilon}} F(x)\Leftrightarrow \mathcal{T}_{\epsilon}(F(x'))\leq\mathcal{T}_{\epsilon}(F(x))\Leftrightarrow F(x)-F(x')\in\mathcal{C}_{\epsilon}\backslash\{0\},
\end{align*}
which is a contradiction to the fact that $x$ is an $\epsilon$-properly Pareto optimal solution.

($\Leftarrow$) Assume that for $x\in\Omega$, there does not exist $x'\in\Omega$, such that $\mathcal{T}_{\epsilon}(F(x'))\leq \mathcal{T}_{\epsilon}(F(x))$. To the contrary, suppose that $x$ is not an $\epsilon$-properly Pareto optimal solution. According to Definition \ref{de:1}, we know that there exists $\bar{x}\in\Omega$, such that $F(x)\in F(\bar{x})+\mathcal{C}_{\epsilon}\backslash\{0\}$. By Lemma \ref{le:2} and the linearity of $\mathcal{T}_{\epsilon}$, we obtain a contradiction $\mathcal{T}_{\epsilon}(F(\bar{x}))\leq \mathcal{T}_{\epsilon}(F(x))$. \qed
\end{proof}

Now we can propose the discarding test corresponding to $\mathcal{C}_{\epsilon}$.

\noindent{\bf $\mathcal{C}_{\epsilon}$-dominance-based discarding test} \emph{Let problem \eqref{VP} be given, let $B$ be a subbox and $l(B)$ its lower bound. For $\epsilon\in[0,1)$, if there exists a feasible objective vector $u\in F(\Omega)$, such that $u\leqq_{\mathcal{C}_{\epsilon}} l(B)$, then $B$ will be discarded.}

The correctness of the proposed discarding test is stated next.

\begin{theorem}\label{th:2}
Let a subbox $B\in\Omega$ and its lower bound $l(B)\in\mathbb{R}^m$ be given. For $\epsilon\in[0,1)$, let $\mathcal{U}(\mathcal{C}_{\epsilon})$ be a nondominated upper bound set of problem \eqref{VP} with respect to $\mathcal{C}_{\epsilon}$. If there exists an upper bound $u\in\mathcal{U}(\mathcal{C}_{\epsilon})$ such that $u\leqq_{\mathcal{C}_{\epsilon}} l(B)$, then $B$ does not contain $\epsilon$-properly Pareto optimal solution of problem \eqref{VP}.
\end{theorem}
\begin{proof}
Let us assume that $x^*\in B$ is an $\epsilon$-properly Pareto solution of problem \eqref{VP}. According to \eqref{IE:2.2}, we know that $l(B)\leq F(x^*)$. If $u=l(B)$, then $u\leq F(x^*)$. By Lemma \ref{le:1}, we have $u\leq_{\mathcal{C}_{\epsilon}} F(x^*)$. If $u\leq_{\mathcal{C}_{\epsilon}}l(B)$, by Lemma \ref{le:3}, we then have $u\leq_{\mathcal{C}_{\epsilon}} l(B)\leq_{\mathcal{C}_{\epsilon}} F(x^*)$. Thus from Theorem \ref{th:1}, we obtain $x^*$ is not an $\epsilon$-properly Pareto solution.\qed
\end{proof}

Algorithm \ref{alg:2} gives an implementation of the $\mathcal{C}_{\epsilon}$-dominance-based discarding test, where the flag $D$ stands for decision to discard the subbox after the algorithm. Suppose that a nondominated upper bound set $\mathcal{U}(\mathcal{C}_{\epsilon})$ is known, and if there exists an upper bound $u\in\mathcal{U}(\mathcal{C}_{\epsilon})$ such that $\mathcal{T}_{\epsilon}(u)\leqq \mathcal{T}_{\epsilon}(l(B))$, which means that there dose not exist an $\epsilon$-properly Pareto solution in $B$ by Theorem \ref{th:2}, then the flag $D$ for $B$ is set to 1; otherwise, the flag $D$ is set to 0.

\begin{algorithm}
\caption{\texttt{$\mathcal{C}_{\epsilon}$-dominance-based-DiscardingTest($B,l(B),\mathcal{U},\mathcal{C}_{\epsilon}$)}}\label{alg:2}
    $D\leftarrow0$\;
    \ForEach{$u\in\mathcal{U}$}{
    \If{$\mathcal{T}_{\epsilon}(u)\leqq \mathcal{T}_{\epsilon}(l(B))$}
    {
    $D\leftarrow1$\;
    \textbf{break for-loop}
    }}
    \Return $D$
\end{algorithm}
\subsection{Non-polyhedral ordering cone}
In fact, not all practical problems consider polyhedral cones. For example, in portfolio optimization \cite{ref62}, the dominance structure in the three-dimensional portfolio space is defined by the so-called ice cream cone
\begin{align}
\mathcal{C}:=\{y=(y_1,y_2,y_3)^T:y_3\geq\sqrt{y_1^2+y_2^2}\},\label{E:3.1}
\end{align}
also called the second-order cone. This is a pointed closed convex cone which is non-polyhedral.

We denote the origin of the three-dimensional portfolio space as $O$, and let $z$ be a point on the $z$-axis. It is not difficult to see that the rotation axis of the ice cream cone is $\overrightarrow{Oz}$. For a given point $y\in\mathbb{R}^3$, the left side of the inequality in \eqref{E:3.1} represents the projection of the vector $\overrightarrow{Oy}$ onto $\overrightarrow{Oz}$. The right side of the inequality in \eqref{E:3.1} represents the distance from $y$ to the $z$-axis (the line where $\overrightarrow{Oz}$ is located). Thus, the geometric meaning of the inequality is that the angle between $\overrightarrow{Oy}$ and $\overrightarrow{Oz}$ is not greater than $\pi/4$. As a result, we can determine a unique ice cream cone from the direction and angle. If we extend to the more general case (arbitrary dimension $m\geq2$, direction $w\in\mathbb{R}^m\backslash\{0\}$ and angle $\theta\in(0,\pi/2)$), for two points $y^1,y^2\in\mathbb{R}^m$, we define
\begin{align*}
d_1(y^1,y^2):=\frac{(y^2-y^1)^T w}{\|w\|}~\hbox{and}~d_2(y^1,y^2)=:\|(y^2-y^1)-d_1(y^1,y^2)\frac{w}{\|w\|}\|,
\end{align*}
then a general ice cream cone in $\mathbb{R}^m$ can be expressed as:
\begin{align*}
\mathcal{C}_{(w,\theta)}=\{y\in\mathbb{R}^m:d_2(0,y)\leq d_1(0,y)\tan\theta\}.
\end{align*}
It is easy to see that $\mathcal{C}_{(w,\theta)}$ is a pointed closed convex cone, thus the corresponding cone orderings $\leq_{\mathcal{C}_{(w,\theta)}}$ and $\leqq_{\mathcal{C}_{(w,\theta)}}$ in $\mathbb{R}^m$ can be characterized. Then we have the following lemma.
\begin{lemma}\label{le:4}
For given $w\in\mathbb{R}^m\backslash\{0\}$ and $\theta\in(0,\pi/2)$, if $y\in\mathcal{C}_{(w,\theta)}$, then $d_1(0,y)\geq0$. Especially, $d_1(0,y)=0$ if and only if $y=0$.
\end{lemma}
\begin{proof}
The first conclusion is ensured by the definition of $\mathcal{C}_{(w,\theta)}$ and the fact $d_2(0,y)\geq0$. Next we prove $d_1(0,y)=0\Longleftrightarrow y=0$.\\
$(\Longleftarrow)$ It is obvious by the definition of $d_1(\cdot,\cdot)$.\\
$(\Longrightarrow)$ Assume that $d_1(0,y)=0$. Suppose, to the contrary, that $y\neq0$. On the one hand, by the definition of $d_2(\cdot,\cdot)$, we have $d_2(0,y)=\|y\|>0$. On the other hand, since $y\in\mathcal{C}_{(w,\theta)}$, by the definition of $\mathcal{C}_{(w,\theta)}$, we have that $d_2(0,y)\leq d_1(0,y)\tan\theta=0$. It is a contradiction.\qed
\end{proof}

\begin{theorem}\label{th:3}
Let the direction $w\in\mathbb{R}^m\backslash\{0\}$ and angle $\theta\in(0,\pi/2)$ be given. For $y^1,y^2\in\mathbb{R}^m$, we have $y^1\leqq_{\mathcal{C}_{(w,\theta)}}y^2$ if and only if $d_1(y^1,y^2)\geq0$ and $d_2(y^1,y^2)\leq d_1(y^1,y^2)\tan\theta$.
\end{theorem}
\begin{proof}
($\Longrightarrow$) Since $y^1\leqq_{\mathcal{C}_{(w,\theta)}}y^2$, we have that $y^2-y^1\in\mathcal{C}_{(w,\theta)}$. By the definition of $\mathcal{C}_{(w,\theta)}$, we have $d_2(0,y^2-y^1)\leq d_1(0,y^2-y^1)\tan\theta$. Furthermore, according to Lemma \ref{le:4}, we know that $d_1(0,y^2-y^1)\geq0$. Moreover, by the definitions of $d_1$ and $d_2$, we know that $d_1(y^1,y^2)=d_1(0,y^2-y^1)$ and $d_2(y^1,y^2)=d_2(0,y^2-y^1)$.\\
($\Longleftarrow$) The proof is guaranteed by the definition of $\mathcal{C}_{(w,\theta)}$. \qed
\end{proof}

\begin{corollary}\label{co:1}
Let the direction $w\in\mathbb{R}^m\backslash\{0\}$ and angle $\theta\in(0,\pi/2)$ be given. For $y^1,y^2\in\mathbb{R}^m$, we have $y^1\leq_{\mathcal{C}_{(w,\theta)}}y^2$ if and only if $0\leq d_2(y^1,y^2)/d_1(y^1,y^2)\leq\tan\theta$ and $y^1\neq y^2$.
\end{corollary}
\begin{proof}
The conclusion is derive from Theorem \ref{th:3} and Lemma \ref{le:4}.
\end{proof}

The following theorem states the way to construct a nondominated set with respect to $\mathcal{C}_{(w,\theta)}$.
\begin{theorem}\label{th:4}
Let the direction $w\in\mathbb{R}^m\backslash\{0\}$ and angle $\theta\in(0,\pi/2)$ be given. For two points $y^1,y^2\in\mathbb{R}^m$, we have $y^1\nleqq_{\mathcal{C}_{(w,\theta)}} y^2$ and $y^2\nleqq_{\mathcal{C}_{(w,\theta)}} y^1$ if and only if one of the following conditions holds true:
\begin{enumerate}[\rm (1)]
  \item $d_1(y^1,y^2)<0$ and $d_2(y^2,y^1)/d_1(y^2,y^1)>\tan\theta$;
  \item $d_1(y^2,y^1)<0$ and $d_2(y^1,y^2)/d_1(y^1,y^2)>\tan\theta$.
\end{enumerate}
\end{theorem}
\begin{proof}
First of all, by the definitions of $d_1$ and $d_2$, it is obvious to see that $d_1(y^1,y^2)=-d_1(y^2,y^1)\neq0$ and $d_2(y^2,y^1)=d_2(y^1,y^2)>0$. Therefore, conditions (1) and (2) cannot be satisfied simultaneously. From Theorem \ref{th:3}, we have that
\begin{align*}
y^1\nleqq_{\mathcal{C}_{(w,\theta)}} y^2\Longleftrightarrow d_1(y^1,y^2)< 0~\hbox{or}~d_2(y^1,y^2)/d_1(y^1,y^2)>\tan\theta;\\
y^2\nleqq_{\mathcal{C}_{(w,\theta)}} y^1\Longleftrightarrow d_1(y^2,y^1)< 0~\hbox{or}~d_2(y^2,y^1)/d_1(y^2,y^1)>\tan\theta.
\end{align*}
It follows that, if $d_1(y^1,y^2)< 0$, we know that
\begin{align*}
y^1\nleqq_{\mathcal{C}_{(w,\theta)}} y^2\Longleftrightarrow d_1(y^1,y^2)< 0~\hbox{and}~y^2\nleqq_{\mathcal{C}_{(w,\theta)}} y^1\Longleftrightarrow d_2(y^2,y^1)/d_1(y^2,y^1)>\tan\theta;
\end{align*}
otherwise, if $d_1(y^2,y^1)< 0$, we know that
\begin{align*}
y^2\nleqq_{\mathcal{C}_{(w,\theta)}} y^1\Longleftrightarrow d_1(y^2,y^1)< 0~\hbox{and}~ y^1\nleqq_{\mathcal{C}_{(w,\theta)}} y^2\Longleftrightarrow d_2(y^1,y^2)/d_1(y^1,y^2)>\tan\theta.
\end{align*}\qed
\end{proof}

By Corollary \ref{co:1}, the $\mathcal{C}_{(w,\theta)}$-dominance relation between $x^1,x^2\in\Omega$ can be defined as follows:
\begin{align*}
x^1~\mathcal{C}_{(w,\theta)}\hbox{-}dominates~x^2&\Longleftrightarrow F(x^1)\leq_{\mathcal{C}_{(w,\theta)}}F(x^2)\Longleftrightarrow F(x^2)-F(x^1)\in\mathcal{C}_{(w,\theta)}\backslash\{0\}\\ &\Longleftrightarrow 0\leq \frac{d_2(F(x^1),F(x^2))}{d_1(F(x^1),F(x^2))}\leq\tan\theta,~F(x^1)\neq F(x^2)
\end{align*}

The following theorem reveals the relationship between the solution set $\mathcal{M}(f(\Omega),\mathcal{C}_{(w,\theta)})$ and $\mathcal{C}_{(w,\theta)}$-dominance.
\begin{theorem}\label{th:5}
   Let the direction $w\in\mathbb{R}^m\backslash\{0\}$ and angle $\theta\in(0,\pi/2)$ be given, we have $x\in \mathcal{M}(F(\Omega),\mathcal{C}_{(w,\theta)})$ if and only if there does not exist $x'\in\Omega$, such that $0\leq\frac{d_2(F(x^1),F(x^2))}{d_1(F(x^1),F(x^2))}\leq\tan\theta,~F(x^1)\neq F(x^2)$.
\end{theorem}
\begin{proof}
The proof is guaranteed by Corollary \ref{co:1} and the definition of efficient solution.\qed
\end{proof}

Now we can derive the discarding test with respect to $\mathcal{C}_{(w,\theta)}$.

\noindent{\bf $\mathcal{C}_{(w,\theta)}$-dominance-based discarding test} \emph{Let problem \eqref{VP} be given, let $B$ be a subbox and $l(B)$ its lower bound. We assume that $\mathcal{C}_{(w,\theta)}\supseteq\mathbb{R}^m_+$ for a given direction $w\in\mathbb{R}^m\backslash\{0\}$ and an angle $\theta\in(0,\pi/2)$. If there exists a feasible objective vector $u\in F(\Omega)$, such that $u\leqq_{\mathcal{C}_{(w,\theta)}} l(B)$, then $B$ will be discarded.}

Next we state the correctness of the proposed discarding test.

\begin{theorem}\label{th:6}
Let problem \eqref{VP} be given, let $B$ be a subbox and $l(B)$ its lower bound. We assume that $\mathcal{C}_{(w,\theta)}\supseteq\mathbb{R}^m_+$ where $w\in\mathbb{R}^m\backslash\{0\}$ and $\theta\in(0,\pi/2)$. Let $\mathcal{U}(\mathcal{C}_{(w,\theta)})$ be a nondominated upper bound set with respect to $\mathcal{C}_{(w,\theta)}$. If there exists an upper bound $u\in\mathcal{U}(\mathcal{C}_{(w,\theta)})$ such that $u\leqq_{\mathcal{C}_{(w,\theta)}} l(B)$, then $B\cap \mathcal{M}(F(\Omega),\mathcal{C}_{(w,\theta)})=\emptyset$.
\end{theorem}
\begin{proof}
Let us assume that $x^*\in B\cap \mathcal{M}(f(\Omega),\mathcal{C}_{(w,\theta)})$. According to \eqref{IE:2.2}, we know that $l(B)\leq F(x^*)$. According to Lemma \ref{le:1}, we have $l(B)\leq_{\mathcal{C}_{(w,\theta)}} F(x^*)$. If $u=l(B)$, then we have $u\leq_{\mathcal{C}_{(w,\theta)}} F(x^*)$. If $u\leq_{\mathcal{C}_{(w,\theta)}}l(B)$, by the transitivity of $\mathcal{C}_{(w,\theta)}$, we have $u\leq_{\mathcal{C}_{(w,\theta)}} F(x^*)$. From Theorem \ref{th:5}, we obtain a contradiction $x^*\notin \mathcal{M}(F(\Omega),\mathcal{C}_{(w,\theta)})$.\qed
\end{proof}

Algorithm \ref{alg:3} gives an implementation of the $\mathcal{C}_{(w,\theta)}$-dominance-based discarding test, where the flag $D$ stands for decision to discard the subbox after the algorithm. Suppose that a nondominated upper bound set $\mathcal{U}(\mathcal{C}_{(w,\theta)})$ is known, and if there exists an upper bound $u\in\mathcal{U}(\mathcal{C}_{(w,\theta)})$ such that $u\leqq_{\mathcal{C}_{(w,\theta)}} l(B)$, that is, $d_1(u,l(b))\geq0$ and $d_2(u,l(b))\leq d_1(u,l(b))\tan\theta$, then the flag $D$ for $B$ is set to 1; otherwise, the flag $D$ is set to 0.

\begin{algorithm}
\caption{\texttt{$\mathcal{C}_{(w,\theta)}$-dominance-based-DiscardingTest($B,l(B),\mathcal{U},\mathcal{C}_{(w,\theta)}$)}}\label{alg:3}
    $D\leftarrow0$\;
    \ForEach{$u\in\mathcal{U}$}{
    $d_1\leftarrow \frac{(l(B)-u)^T w}{\|w\|}$, $d_2\leftarrow \|l(B)-u-d_1\frac{w}{\|w\|}\|$\;
    \If{$d_1\geq0$ and $d_2\leq d_1\tan\theta$}
    {
    $D\leftarrow1$\;
    \textbf{break for-loop}
    }}
    \Return $D$
\end{algorithm}

\subsection{The complete algorithm}

The introduction of two new discarding tests allows us to present the cone dominance-based branch and bound algorithm (abbreviated as CBB). As previously stated, the aim of CBB is to approximate a portion of the Pareto optimal set by employing a general cone (either the polyhedral cone $\mathcal{C}_\epsilon$ or the general ice cream cone $\mathcal{C}_{(w,\theta)}$). For the sake of clarity, both types of cones will be referred to as $\mathcal{C}$. It should be noted that the parameters of the cone must satisfy the conditions set forth in the aforementioned theorems. In particular the direction and angle of the general ice cream cone must ensure that the generated cone contains $\mathbb{R}^m_+$. Furthermore, in order to correctly express the region over which the solutions are distributed, the direction of the general ice cream cone is set to $1/w$.

The pseudocode of CBB is presented in Algorithm \ref{alg:4}. It should be noted that a parallel breadth-first search strategy is employed to search all subboxes simultaneously in each iteration. Therefore, in line 4, all subboxes in the box collection $\mathcal{B}_{k-1}$ are bisected perpendicularly to the direction of maximum width in order to construct the current collection $\mathcal{B}_{k}$ simultaneously. This branching process produces subboxes with the same diameter, therefore in line 5, only the diameter of one subbox needs to be calculated to obtain the value of $\omega_k$. Subsequently, in the event that problem \eqref{VP} contains inequality constraints, the feasibility test described in \cite{ref11} is employed to exclude the subboxes that do not contain any feasible solutions.

In the first for-loop, the lower bound $l(B)=(l_1,\ldots,l_m)^T$ of the subbox $B$ is calculated as follows:
\begin{align}
l_i = f_i(m(B))-\frac{L_i}{2}\|\omega(B)\|,\quad i=1,\ldots,m,\label{E:3.2}
\end{align}
where $L_i$ is the Lipschitz constant of $f_i$. Subsequently, all lower bounds are stored in a lower bound set, denoted by $\mathcal{L}$. For each lower bound $l\in \mathcal{L}$, a comparison is made between $l$ and other lower bounds stored in $\mathcal{L}$ via the $\mathcal{C}$-dominance relation: if $l$ is $\mathcal{C}$-dominated by any other lower bounds, then $l$ will be removed from $\mathcal{L}$. Otherwise, the lower bounds that are $\mathcal{C}$-dominated by $l$ will be removed from $\mathcal{L}$. This process allows us to identify a nondominated lower bound set with respect to $\mathcal{C}$, denoted by $\mathcal{L}_k(\mathcal{C})$, which can be extracted from the original set of lower bounds $\mathcal{L}$. The subboxes corresponding to the lower bounds stored in $\mathcal{L}_k(\mathcal{C})$ constitute $\bar{\mathcal{B}}$.

In the second for-loop, an MOEA is employed to determine the upper bounds for all subboxes stored in $\bar{\mathcal{B}}$. The upper bounds and the corresponding solutions (the preimage of the upper bounds) are stored in the sets $\mathcal{U}$ and $\mathcal{X}$, respectively. In order to reduce the computational costs, a mini MOEA is applied, which has a small initial population size and a few generations. In the event that the problem also contains inequality constraints, it is possible to employ constrained handling approaches, as outlined in \cite{ref15,ref46}, in order to guarantee that feasible solutions and upper bounds are obtained. Thereafter, the nondominated upper bound set with respect to $\mathcal{C}$, denoted by $\mathcal{U}_k(\mathcal{C})$, and its corresponding solution set, denoted by $\mathcal{X}_k(\mathcal{C})$, will be identified from the sets $\mathcal{U}$ and $\mathcal{X}$ by means of the $\mathcal{C}$-dominance relation.

In the third for-loop, the discarding flags for all subboxes are calculated by Algorithms \ref{alg:2} or \ref{alg:3} and are stored in a flag list $\mathcal{D}$. Subsequently, in line 22, the subboxes whose flags are equal to 1 will be removed from $\mathcal{B}_k$. It is also worth noting that, in order to accelerate the computational process, it is possible to compute the discarding flags for all subboxes simultaneously, that is to say, to parallelize the third for-loop. In addition, the first and second for-loops can also be parallelized.

%executes the parallel probabilistic search strategy on $\mathcal{B}_{k}$, and outputs the lower bound set $\mathcal{L}_k$, the upper bound set $\mathcal{U}$ and the solution set $\mathcal{X}$. Thereafter, the $\epsilon$-dominance relation is applied to $\mathcal{U}$ to find the non-$\epsilon$-dominated upper bound set $\mathcal{U}_k$. If the iteration count $k$ is a multiple of $2n$, a reboot operation will be performed to set the probabilities of all subboxes to 1; otherwise, the number of upper bounds for each subbox in $\mathcal{U}_k$ will be counted and the exploration probabilities of all subboxes will be updated according to the formula \eqref{E:3.3}. In the second for-loop, according to the non-$\epsilon$-dominated upper bound set $\mathcal{U}_k$ and the lower bound set $\mathcal{L}_k$, the subboxes which do not contain $\epsilon$-properly Pareto optimal solutions will be discarded by the $\epsilon$-discarding test.

\begin{algorithm}
\caption{\texttt{Cone Dominance-Based Branch and Bound Algorithm}}\label{alg:4}
  \SetKwInOut{Input}{Input}\SetKwInOut{Output}{Output}
  \SetKwFunction{MOEA}{MOEA}
  \SetKwFunction{DT}{$\mathcal{C}$-dominance-based-DiscardingTest}
  \Input{problem \eqref{VP}, cone $\mathcal{C}$;}
    \Output{$\mathcal{B}_{k}$, $\mathcal{U}_{k}(\mathcal{C})$, $\mathcal{X}_k(\mathcal{C})$;}
    $k\leftarrow1$, $\mathcal{B}_0\leftarrow\Omega$, $\omega_{k-1}\leftarrow \|\omega(\Omega)\|$, $d\leftarrow10^6$\;
  \While{$d>\varepsilon$ or $\omega_{k-1}>\delta$}{
  $\mathcal{L}\leftarrow\emptyset$, $\mathcal{U}\leftarrow\emptyset$, $\mathcal{X}\leftarrow\emptyset$\;
  Construct $\mathcal{B}_{k}$ by bisecting all boxes in $\mathcal{B}_{k-1}$\;
  $\omega_k\leftarrow \max\{\|\omega(B)\|:B\in\mathcal{B}_{k}\}$\;
  Update $\mathcal{B}_{k}$ by the feasibility test suggested in \cite{ref11}\;
  \ForEach{$B\in\mathcal{B}_{k}$}
  {
    Calculate for $B$ its lower bound $l(B)$ by equation (\ref{E:3.2})\;
    %Update $\mathcal{L}_k$ by $l(B)$ according to the $\epsilon$-dominance\;
    $\mathcal{L}\leftarrow \mathcal{L}\cup l(B)$\;
  }
  Find a nondominated lower bound set $\mathcal{L}_k(\mathcal{C})$ from $\mathcal{L}$ by the $\mathcal{C}$-dominance\;
  Determine the box collection $\bar{\mathcal{B}}\subset\mathcal{B}_k$ according to $\mathcal{L}_k(\mathcal{C})$\;
  \ForEach{$B\in\bar{\mathcal{B}}$}
  {
    $\bar{\mathcal{U}},\bar{\mathcal{X}}\leftarrow\MOEA(B)$\;
    $\mathcal{U}\leftarrow\mathcal{U}\cup \bar{\mathcal{U}}$, $\mathcal{X}\leftarrow\mathcal{X}\cup \bar{\mathcal{X}}$\;
  }

  Find a nondominated upper bound set $\mathcal{U}_k(\mathcal{C})$ from $\mathcal{U}$ and the corresponding solution set $\mathcal{X}_k(\mathcal{C})$ from $\mathcal{X}$ by the $\mathcal{C}$-dominance\;
  %Find the corresponding solution set $\mathcal{X}_k\subseteq\mathcal{X}$\;
  \ForEach{$B\in\mathcal{B}_k$}
  { $D(B)\leftarrow \DT(B,l(B),\mathcal{U}_k(\mathcal{C}))$\;
  $\mathcal{D}\leftarrow \mathcal{D}\cup D(B)$\;
  }
  Update $\mathcal{B}_k$ according to the flag list $\mathcal{D}$\;
  $d\leftarrow d_h(\mathcal{U}_{k}(\mathcal{C}),\mathcal{L}_{k}(\mathcal{C}))$, $k\leftarrow k+1$.
  }
\end{algorithm}

In order to ensure that the trade-offs between disparately scaled objectives can be correctly expressed, we use an objective normalization technique to replace $f_i\;(i=1,\ldots,m)$ by
\begin{align*}
\bar{f}_i=\frac{f_i-z^*_i}{z^{nad}_i-z^*_i},
\end{align*}
where $z^{nad}=(z^{nad}_1,\ldots,z^{nad}_m)^T$ and $z^*=(z^*_1,\ldots,z^*_m)^T$ are the nadir and ideal points, respectively, i.e., $z^*=\min\{f_i(x):x\in\mathcal{M}(f(\Omega),\mathbb{R}^m_+)\}$ and $z^{nad}=\max\{f_i(x):x\in\mathcal{M}(f(\Omega),\mathbb{R}^m_+)\}$. In other words, $z^*$ and $z^{nad}$ represent the lower and upper bounds of the Pareto front, respectively. It is not straightforward or obligatory to compute $z^*$ and $z^{nad}$. In our implementation, we replace $z^*_i$ by $l^*_i$, which is the smallest value in current lower bound set $\mathcal{L}$. This is defined as $l^*_i=\min\{l_i:l\in \mathcal{L}\}$. Furthermore, it should be noted that the lower bounds are calculated by the midpoints of the subboxes. Consequently, the images of all the midpoints of the subboxes can be used to construct the set $M=\{F(m(B)):B\in\mathcal{B}_k\}$, and then identify a nondominated set $M(\mathbb{R}^m_+)$ from $M$ via the Pareto dominance. The value $z^{nad}$ is replaced by $u^{nad}_i$, which represents the largest value in $M(\mathbb{R}^m_+)$. This is defined as $u^{nad}_i=\max\{u_i:u\in M(\mathbb{R}^m_+)\}$. Both $u^{nad}$ and $l^*$ can be obtained in the first for-loop. It is of paramount importance to note that if the value on $i$-th coordinate of the lower bound (or the image of the midpoint) of a subbox is equal to $l^*_i$ (or $u^{nad}_i$), then the subbox will not be removed in the current iteration, even if its lower bound is $\mathcal{C}$-dominated by an upper bound. Furthermore, if the values on the $i$-th coordinate of the lower bounds (or the images of midpoints) of several subboxes are all equal to $l^*_i$ (or $u^{nad}_i$), all of these subboxes will not be removed. As a result, before applying the $\mathcal{C}$-dominance or $\mathcal{C}$-dominance based discarding test, it is then possible to normalize both the upper bound $u$ and lower bound $l$ as follows:
\begin{align*}
\bar{u}_i=\frac{u_i-l^*_i}{u^{nad}_i-l^*_i}\quad and\quad\bar{l}_i=\frac{l_i-l^*_i}{u^{nad}_i-l^*_i},\,i=1,\ldots,m.
\end{align*}
Furthermore, if the problem contains inequality constraints, only the feasible midpoints are employed in the computation of $u^{nad}$. It is possible that every midpoint is infeasible. In such a case, we are able to identify a set of random points within each subbox and then select the feasible ones to constitute $u^{nad}$.

%There are two ways to obtain the nadir and ideal points: one is to solve each objective as a single-objective optimization problem to obtain the extreme points of the Pareto front. In this case, we should choose a global algorithm to solve the single-objective optimization problems; the second approach is to update $z^{nad}$ and $z^*$ during the solution process. The initial values of $z^{nad}$ and $z^*$ can be obtained from the natural interval extension of the objectives. The lower bound of the natural interval expansion is $z^*$, and the upper bound is $z^{nad}$. During the iterations, the minimum value of $f_i$ in the current lower bounds is used to update $z^*_i$, while the maximum value of $f_i$ in the current upper bounds is used to update $\tilde{z}^{nad}$. If a subbox can search for the maximum value of $f_i$ in the upper bounds or the minimum value of $f_i$ in the lower bounds, then it will be stored in the current box collection and not be removed by the $\epsilon$-discarding test.

\subsection{Convergence results}
We start by showing the termination of the algorithm.
\begin{theorem}\label{th:7}
Let the predefined parameters $\varepsilon>0$ and $\delta>0$ be given, CBB terminates after a finite number of iterations.
\end{theorem}
\begin{proof}
Because we divide all boxes perpendicular to a side with maximal width, $\omega_k$ decreases among the sequence of box collections, i.e., $\omega_{k} > \omega_{k+1}$ for every $k$ and converges to 0. Therefore, for a given $\delta>0$, there must exist a iteration count $\tilde{k}>0$ such that $\omega_{\tilde{k}}\leq\delta$.

Assume $\mathcal{U}_{k}(\mathcal{C})$ and $\mathcal{L}_{k}(\mathcal{C})$ are upper and lower bound set generated by CBB, respectively, and we use (\ref{E:3.2}) to calculate lower bounds. According to the way $\mathcal{U}_{k}(\mathcal{C})$ is constructed and the $\mathcal{C}$-dominance-based discarding test, for every $u\in\mathcal{U}_{k}(\mathcal{C})$, there exists a subbox $B\in\mathcal{B}_k$, such that $F^{-1}(u)\in B$ and $l(B)\in\mathcal{L}_k(\mathcal{C})$. Then, based on the Lipschitz condition, we have
\begin{align*}
%d(u,\mathcal{L}_k)\leq d(u,l(B))= \frac{1}{2}\omega_k\|L\|,\label{eq1}
d(u,\mathcal{L}_k(\mathcal{C}))\leq d(u,l(B))= \|u-(F(m(B))+\frac{L}{2}\omega_k)\|&\leq\|u-F(m(B)\|+\frac{1}{2}\omega_k\|L\|\\
&\leq\omega_k\|L\|,
\end{align*}
where $L=(L_1,\ldots,L_m)^T$ consisting of the Lipschitz constants of objectives. Hence we have
\begin{align}
d_h(\mathcal{U}_{k}(\mathcal{C}),\mathcal{L}_{k}(\mathcal{C})) = \omega_k\|L\|.\label{E:3.3}
\end{align}
Due to the fact that $\omega_k$ converges to 0, it follows that for a given $\varepsilon>0$, there must exist a iteration count $\bar{k}>0$ such that $d_h(\mathcal{U}_{k}(\mathcal{C}),\mathcal{L}_{k}(\mathcal{C}))\leq\varepsilon$.\qed
\end{proof}

The next theorem states all efficient solutions of problem \eqref{VP} with respect to the cone $\mathcal{C}$ are contained in the union of subboxes generated by the algorithm.
\begin{theorem}\label{th:8}
Let $\mathcal{C}\subseteq\mathbb{R}^m$ be a pointed closed convex cone and satisfy $\mathcal{C}\supseteq\mathbb{R}^m_+$. Let $\mathcal{M}(F(\Omega),\mathcal{C})$ be an efficient solution set of problem \eqref{VP} with respect to $\mathcal{C} $ and $\{\mathcal{B}_k\}_{k\in\mathbb{N}}$ a sequence of box collections generated by CBB. Then, for arbitrary $k\in\mathbb{N}$ we have $\mathcal{M}(F(\Omega),\mathcal{C})\subseteq\bigcup_{B\in\mathcal{B}_k} B$.
\end{theorem}
\begin{proof}
Let us suppose that there exists $k\in\mathbb{N}$ and an efficient solution $\bar{x}^*\in \mathcal{M}(F(\Omega),\mathcal{C})$ such that $\bar{x}^*\notin\bigcup_{B\in\mathcal{B}_k} B$, meaning that $\bar{x}^*$ is in a removed subbox $B$ in the pervious iteration. From the $\mathcal{C}$-dominance-based discarding test, we know that $B$ will be discarded if and only if there exists a feasible objective vector $u$ such that $u\leqq_{\mathcal{C}} l(B)$. According to \eqref{IE:2.2}, we then have $l(B)\leq_{\mathcal{C}} F(\bar{x}^*)$. Thus we know that $u\leq_{\mathcal{C}} F(\bar{x}^*)$, which contradicts the assumption that $\bar{x}^*\in \mathcal{M}(F(\Omega),\mathcal{C})$.\qed
\end{proof}

\begin{corollary}
Let $\mathcal{C}\subseteq\mathbb{R}^m$ be a pointed closed convex cone, which satisfies $\mathcal{C}\supseteq\mathbb{R}^m_+$. Let $\mathcal{M}(F(\Omega),\mathcal{C})$ be an efficient solution set of problem \eqref{VP} with respect to the cone $\mathcal{C} $ and $\{\mathcal{B}_k\}_{k\in\mathbb{N}}$ a sequence of box collections generated by CBB. For each $x^*\in \mathcal{M}(F(\Omega),\mathcal{C})$ and $k\in\mathbb{N}$, these exists $B(x^*,k)\in\mathcal{B}_k$, such that $x^*\in B(x^*,k)$. Furthermore, we have $\lim\limits_{k\rightarrow \infty}d_H(\mathcal{M}(F(\Omega),\mathcal{C}),\bigcup_{x\in \mathcal{M}(F(\Omega),\mathcal{C})}B(x,k))=0$.
\end{corollary}
\begin{proof}
The first conclusion is guaranteed by Theorem \ref{th:8}. Next we would like to prove the second conclusion.

On the one hand, from the first conclusion, we have
\begin{align*}
d(x^*,\bigcup_{x\in \mathcal{M}(F(\Omega),\mathcal{C})}B(x,k)) =0, \quad x^*\in \mathcal{M}(F(\Omega),\mathcal{C}),~k\in\mathbb{N},
\end{align*}
it follows $\lim\limits_{k\rightarrow \infty}d_h(\mathcal{M}(F(\Omega),\mathcal{C}),\bigcup_{x\in \mathcal{M}(F(\Omega),\mathcal{C})}B(x,k))=0$.

On the other hand, for each $x\in\bigcup_{x\in \mathcal{M}(F(\Omega),\mathcal{C})}B(x,k)$, there exists $B(\hat{x}^*,k)\in\bigcup_{x\in \mathcal{M}(F(\Omega),\mathcal{C})}B(x,k)$, such that $x\in B(\hat{x}^*,k)$. We then have
\begin{align*}
0\leq\lim\limits_{k\rightarrow \infty}d(x,\mathcal{M}(F(\Omega),\mathcal{C}))\leq\lim\limits_{k\rightarrow \infty}d(x,\hat{x}^*)\leq\lim\limits_{k\rightarrow \infty} w_k =0,
\end{align*}
meaning that $\lim\limits_{k\rightarrow \infty}d_h(\bigcup_{x\in \mathcal{M}(F(\Omega),\mathcal{C})}B(x,k), \mathcal{M}(F(\Omega),\mathcal{C}))=0$. The second conclusion is proven. \qed
\end{proof}

In the next theorem, we prove that the solution set generated by Algorithm \ref{alg:4} is an $\varepsilon e$-efficient solution set of problem \eqref{VP}, where $e$ denotes the $m$-dimensional all-ones vector $(1,\ldots,1)^T\in\mathbb{R}^m$.

\begin{theorem}\label{th:9}
Let $\mathcal{C}\subseteq\mathbb{R}^m$ be a pointed closed convex cone and satisfy $\mathcal{C}\supseteq\mathbb{R}^m_+$. For given predefined parameters $\varepsilon>0$ and $\delta>0$. Assume that the upper bounds are the images of the midpoints of subboxes. Let $\mathcal{X}(\mathcal{C})$ be the solution set with respect to $\mathcal{C}$ generated by CBB and $\mathcal{L}_k(\mathcal{C})$ the nondominated lower bound set with respect to $\mathcal{C}$. Then we have $\mathcal{X}(\mathcal{C})$ is an $\varepsilon e$-efficient solution set of problem \eqref{VP} with respect to $\mathcal{C}$.
\end{theorem}
\begin{proof}
  It is easy to see $\varepsilon e\in\mathcal{C}\backslash\{0\}$. Assume that $\tilde{x}\in \mathcal{X}(\mathcal{C})$ is the midpoint of the subbox $B\in\mathcal{B}_k$ and $l\in\mathcal{L}_k(\mathcal{C})$ is corresponding lower bound. According to \eqref{E:3.3}, we know that
  \begin{align}
    \varepsilon\geq \frac{1}{2}\omega_k\|L\|>\frac{1}{2}\omega_k L_{{\rm max}},\label{IE:3.4}
  \end{align}
  where $L_{max}=\max\{L_i,i=1,\dots,m\}$. Then we can obtain a lower bound $\tilde{l}=(\tilde{l}_1,\ldots,\tilde{l}_m)^T$ whose component can be calculated by
  \begin{align*}
    \tilde{l}_i=f(\tilde{x})_i-\frac{1}{2}\omega_k L_{{\rm max}},
  \end{align*}
  and further, it is easy to see that $F(\tilde{x})-\varepsilon e <\tilde{l}\leqq l$.

  In the following we will prove $\tilde{x}$ is an $\varepsilon$-efficient solution of problem \eqref{VP} in two aspects. On the one hand, by \eqref{IE:2.2}, we have
  \begin{align*}
    F(\tilde{x})-\varepsilon e <\tilde{l}\leqq l\leq F(x),\quad x\in B,
  \end{align*}
  following $F(\tilde{x})-\varepsilon e <_{\mathcal{C}}F(x)$ by Lemma \ref{le:1}. Therefore, there does not exist $x\in B$ with $F(x)\leqq_{\mathcal{C}} F(\tilde{x})-\varepsilon e$.

  On the other hand, suppose that there exists another subbox $B'\in\mathcal{B}_k\backslash B$ and a feasible point $x'\in B'$ such that $F(x')\leqq_{\mathcal{C}} F(\tilde{x})-\varepsilon e$, and thus we have
  \begin{align*}
  F(x')-\frac{1}{2}\omega_k\|L\|e\leq_\mathcal{C} F(x')\leqq_\mathcal{C} F(\tilde{x})-\varepsilon e.
  \end{align*}
  By Lemma \ref{le:1}, we have
    \begin{align*}
    F(\tilde{x})-\varepsilon e<_\mathcal{C}\tilde{l}\leqq_\mathcal{C} l,
  \end{align*}
  following $F(x')-\frac{1}{2}\omega_k\|L\|e<_\mathcal{C} l,$ which is a contradiction to the fact $\mathcal{L}_k(\mathcal{C})$ is a nondominated lower bound set with respect to $\mathcal{C}$. Therefore, for any subbox $B'\in\mathcal{B}_k\backslash B$, there does not exist a point $x'\in B'$ with $F(x')\leqq_\mathcal{C} F(\tilde{x})-\varepsilon e$. Now, we prove the conclusion. \qed
\end{proof}

%To make it easier to obtain the conclusion in Theorem \ref{th:5}, we consider only the situation where the upper bounds are computed by using the midpoints, rather than mini MOEA. However, in practice, the tightness of the upper bounds obtained by mini MOEA tends to be better than  the ones calculated by midpoints. Therefore, the accuracy of the solution set generated by Algorithm \ref{alg:3} will be higher than the accuracy of the one discussed in Theorem \ref{th:5}.

\section{Experimental Results}
%The present section provides a detail description of Algorithm \ref{alg:3} including the motivation, implementation and convergence analysis. Note that we neither aim at studying the effect of parameters on KBB, nor do we intend to compare the performance of KBB to other branch and bound algorithms. In contrast, We aim to propose a practical and convergent algorithm for preference-based multiobjective optimization.

CBB is implemented in Python 3.8 with fundamental packages such as numpy, scipy and multiprocessing. It is executed on a computer with Intel(R) Core(TM) i7-10700 CPU and 32 GB of RAM, running the Windows 10 Professional operating system. In all experiments, MOEA/D \cite{ref20}  is employed with a population size of 10 and a generation number of 20 in CBB. For the polyhedral cone $\mathcal{C}_{\epsilon}$, we set the value of $\epsilon$ to 0.75 and refer to it as $\mathcal{C}_{0.75}$. For the general ice cream cone $\mathcal{C}_{(w,\theta)}$, we set $\theta_1=\arccos\frac{1-\epsilon}{\sqrt{m(m-1)\epsilon^2+m(1+(m-2)\epsilon)^2}}$ or $\theta_2=\arccos\frac{(m-1)(1-\epsilon)}{\sqrt{m(m-1)+m(m-1)^2\epsilon^2}}$, where $\epsilon=0.75$ and $m$ is the number of objectives. As a result, if we set $w=(0.5,0.5,\ldots,0.5)^T\in\mathbb{R}^m$, then the cross sections of $\mathcal{C}_{(w,\theta_1)}$ and $\mathcal{C}_{(w,\theta_2)}$ are the circumscribed and inscribed circles of the cross section of $\mathcal{C}_{\epsilon}$, respectively. Furthermore, all Pareto optimal sets of the following problems are approximated by CBB with $\mathcal{C}_{0}$, that is, we set $\epsilon=0$.

%It is worth noting that we do not intend to compare Algorithm \ref{alg:3} with other algorithms. This is because the aim of most branch and bound algorithms is to search for Pareto optimal solution set rather than $\epsilon$-properly Pareto solution set. Moreover, although some scalarization-based methods can obtain $\epsilon$-properly Pareto solutions, not only they are not easily applied to some nonconvex problems, but also . Therefore, in the following we only demonstrate the effectiveness of Algorithm \ref{alg:3} on some test problems as well as engineering constrained optimization problems.

\subsection{Test problems DEB2DK and DEB3DK}
First, we consider the 2-objective 5-variable DEB2DK and 3-objective 3-variable DEB3DK problems \cite{ref3}. Both problems use a parameter $K$ to regulate the number of bulges in the Pareto front. In the course of our experiments, we set the parameter $K$ to 4 for DEB2DK and 1 for DEB3DK. In DEB2DK, the precision parameters are set to $(\varepsilon,\delta)=(0.0015,0.00015)$, while in DEB3DK, they are set to $(\varepsilon,\delta)=(0.006,0.008)$. The experimental results are presented in Fig.~\ref{fig2}. The Pareto fronts (blue dots) are identified by CBB using the cone $\mathcal{C}_0$. The red stars represent the upper bounds obtained by CBB with the cone $\mathcal{C}_{0.75}$. It is evident that the upper bounds generated by CBB are concentrated in the bulge of the Pareto front. In the theoretical analysis, the algorithm searches for $\epsilon$-properly Pareto optimal solutions with bounded trade-offs, which are distributed exactly on the bulge of the Pareto front.

\begin{figure}[htbp]
\centering
\subfigure[DEB2DK]{
\includegraphics[width=0.45\textwidth]{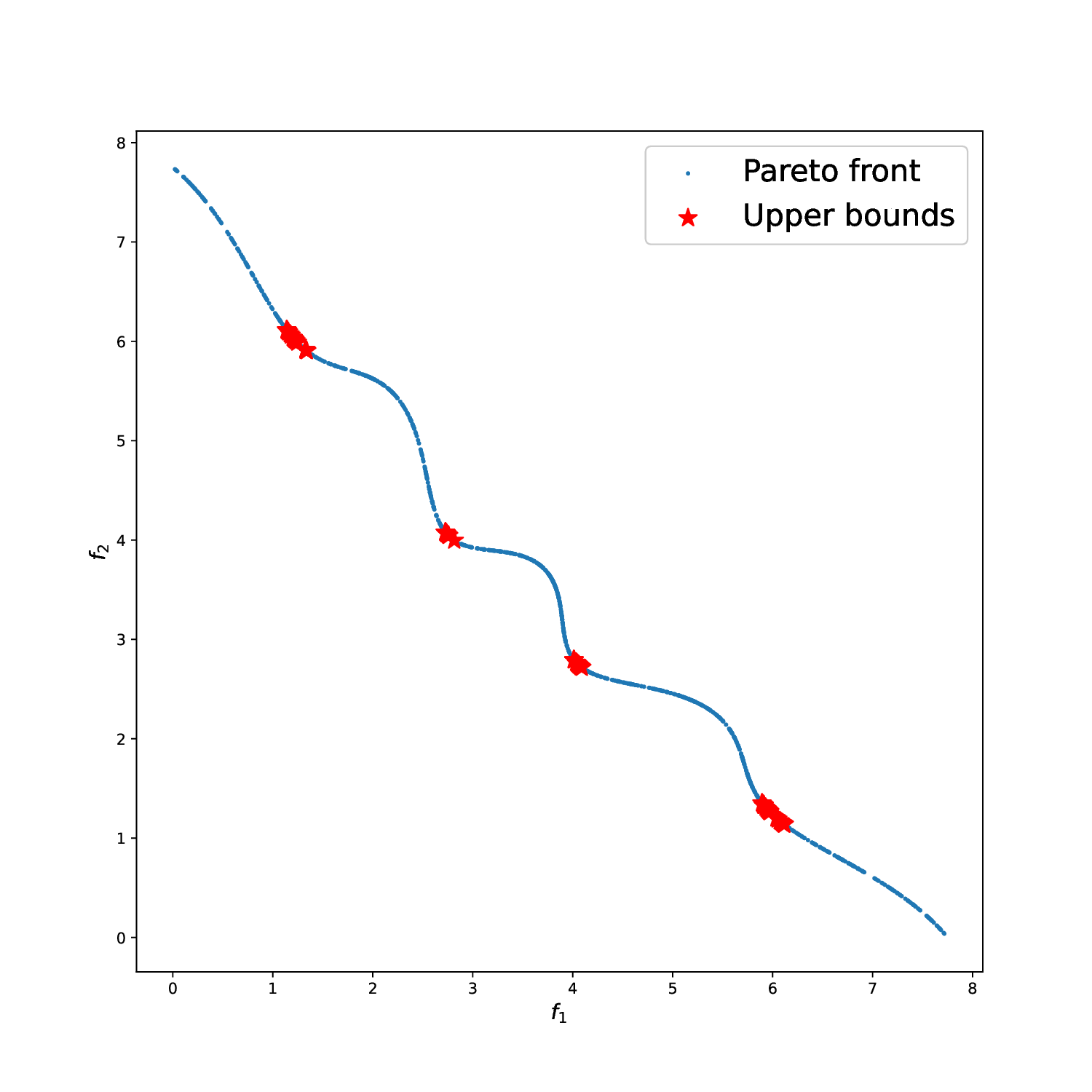}}
\subfigure[DEB3DK]{
\includegraphics[width=0.45\textwidth]{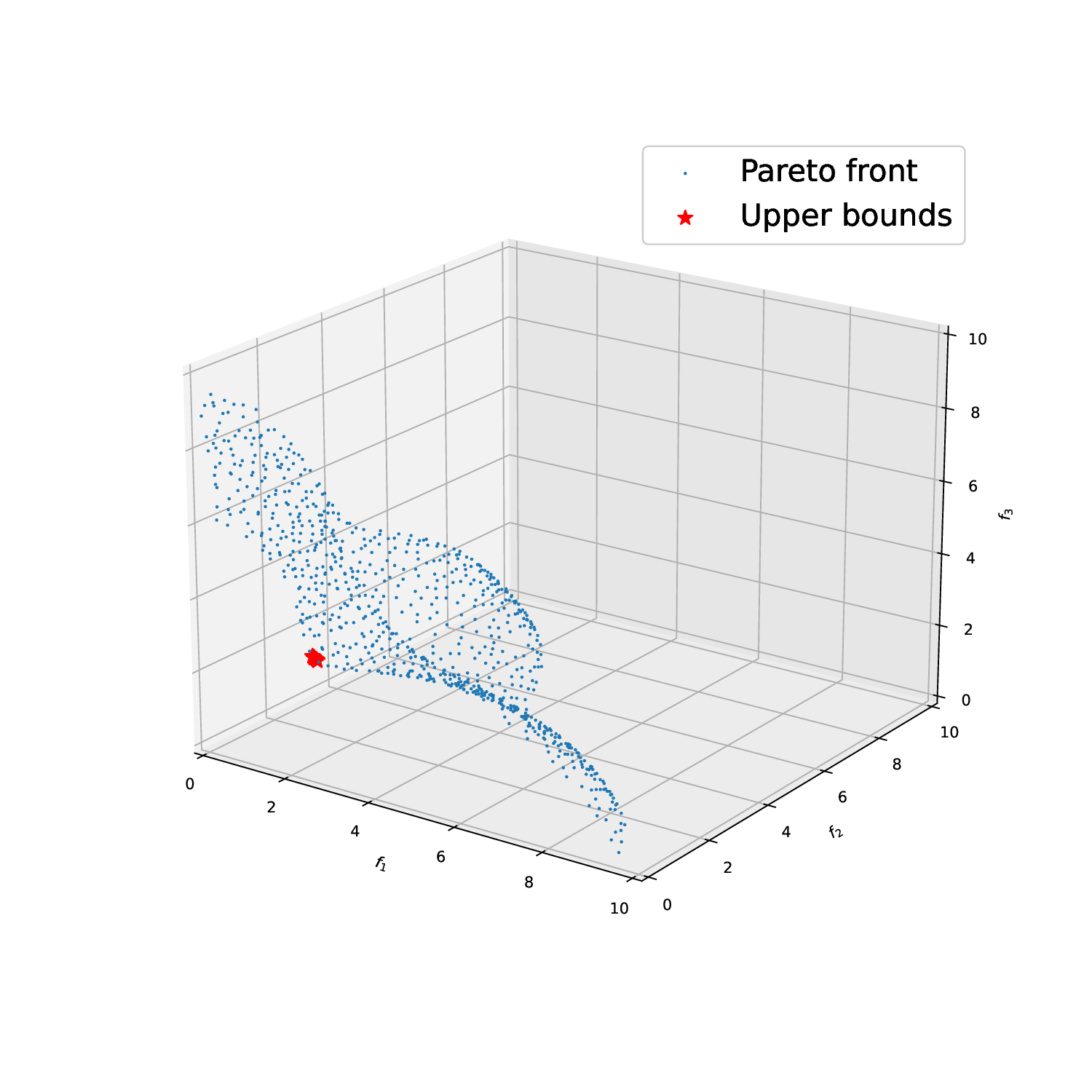}}
\caption{Results of CBB on DEB2DK and DEB3DK.}\label{fig2}
\end{figure}

\subsection{Scaled problems}

The modified test problems from \cite{ref37} are considered next. One of the problem TP1 is given as follows:
$$F(x)=
\begin{pmatrix}
k_1(x_1-1)^2+k_1(x_2-1)^2\\
k_2(x_1+1)^2+k_2(x_2+1)^2
\end{pmatrix},x\in[-2,2]^2.$$
TP2 is given as follows:
$$F(x)=
\begin{pmatrix}
0.5k_1(\sqrt{1+(x_1+x_2)^2}+\sqrt{1+(x_1-x_2)^2+x_1-x_2})+k_1e^{-(x_1-x_2)^2}\\
0.5k_2(\sqrt{1+(x_1+x_2)^2}+\sqrt{1+(x_1-x_2)^2-x_1+x_2})+k_2e^{-(x_1-x_2)^2}
\end{pmatrix},x\in[-1.5,1.5]^2.$$

To test the effectiveness of the normalization technique, we multiplied the two objectives of each problem by the scale factors $k_1=0.1$ and $k_2=10$, respectively. We still use CBB with $\mathcal{C}_{0.75}$ to solve these two problems. The experimental results are shown in Fig.~\ref{fig3}. Unsurprisingly, CBB accurately finds the bulge in the Pareto front even though the objectives are scaled. Furthermore, Figs. \ref{fig3}(c) and \ref{fig3}(d) show the number of subboxes produced by CBB with $\mathcal{C}_{0.75}$ and CBB with $\mathcal{C}_{0}$ in each iteration on two problems, respectively. It is easy to see that when the $\mathcal{C}_{0.75}$ dominance-based discarding test is used, the number of subboxes is significantly reduced in each iteration, meaning that the computational cost of CBB with $\mathcal{C}_{0.75}$ is low.

\begin{figure}[htbp]
\centering
\subfigure[TP1]{
\includegraphics[width=0.45\textwidth]{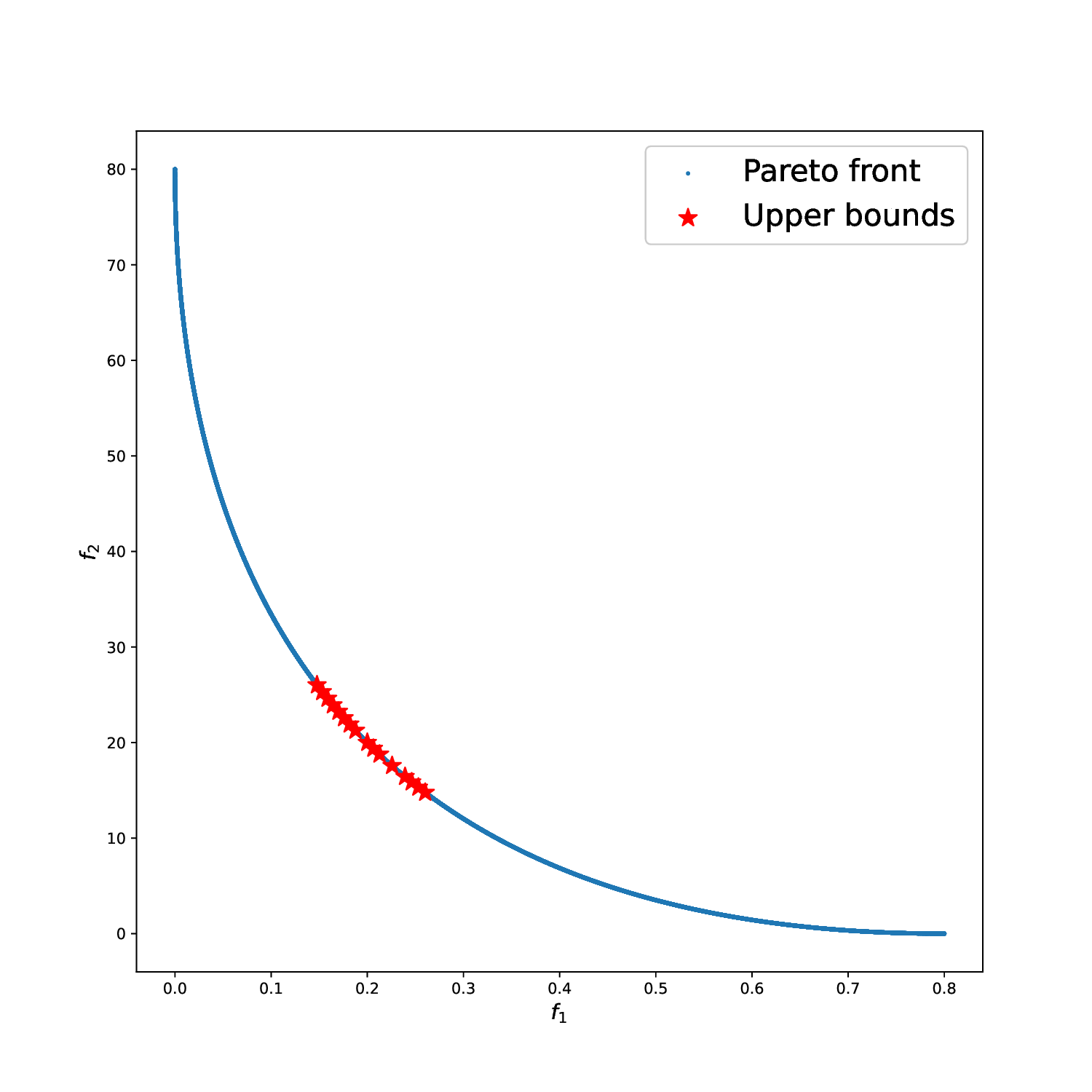}}
\subfigure[TP2]{
\includegraphics[width=0.45\textwidth]{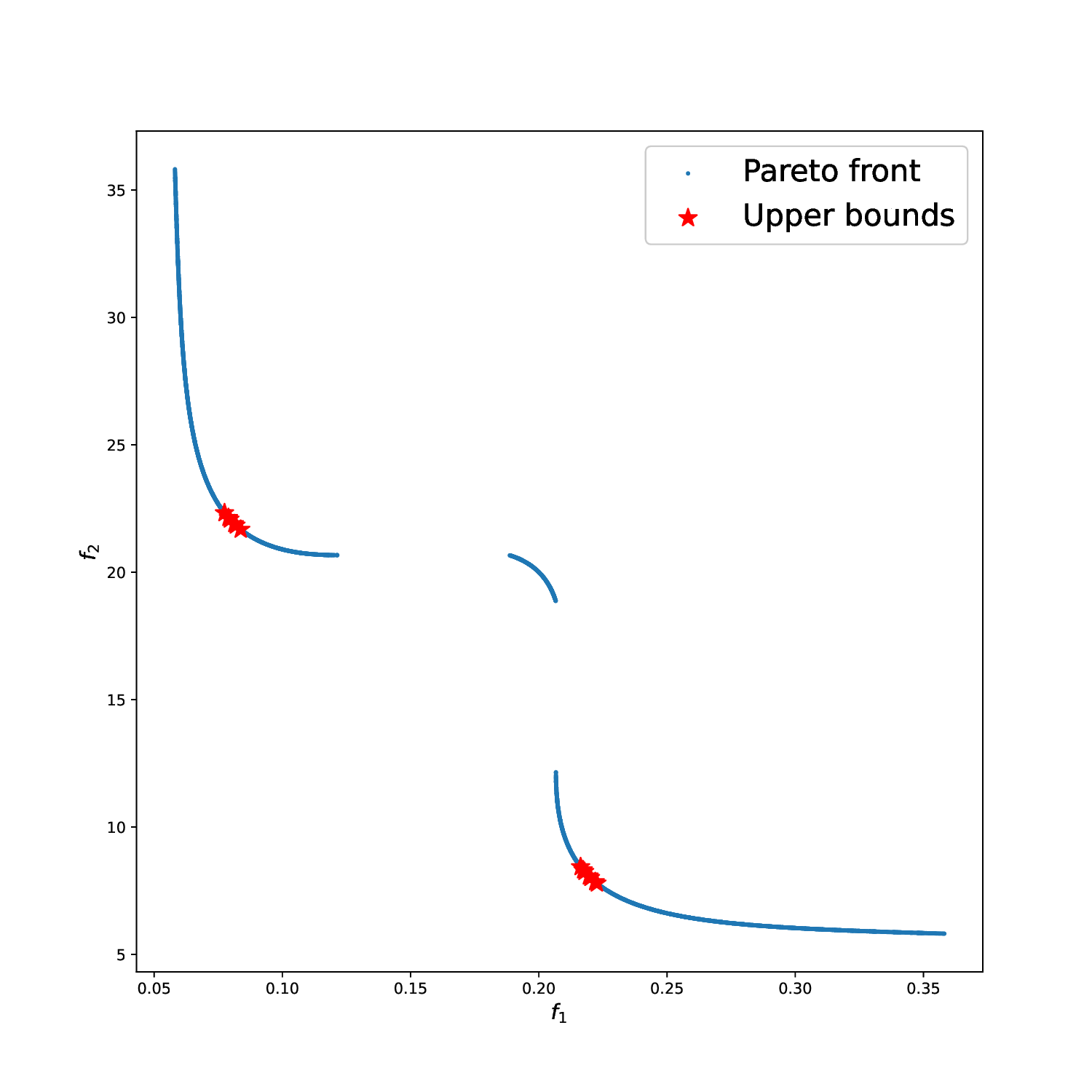}}
\subfigure[the curves of the number of subboxes on TP1]{
\includegraphics[width=0.45\textwidth]{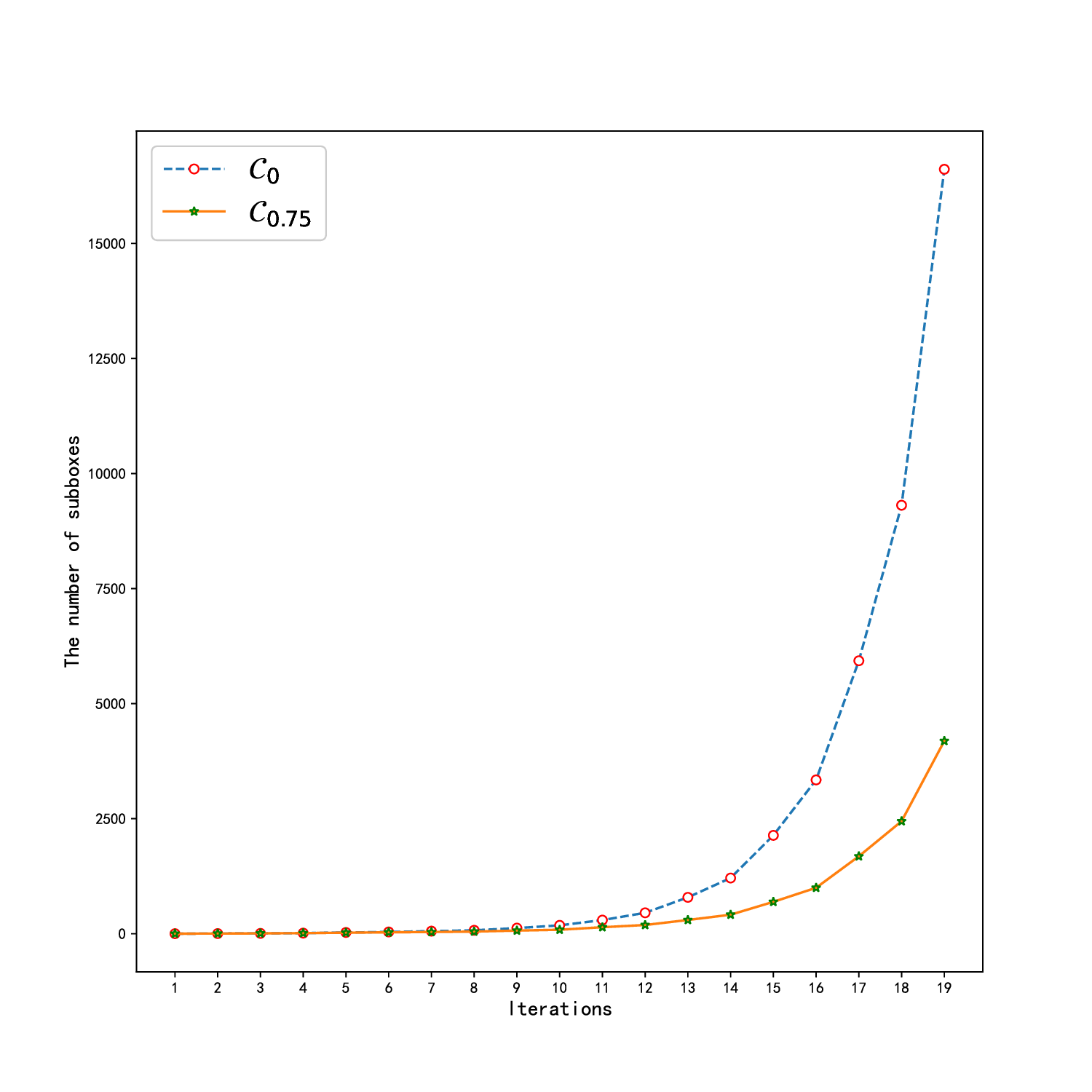}}
\subfigure[the curves of the number of subboxes on TP2]{
\includegraphics[width=0.45\textwidth]{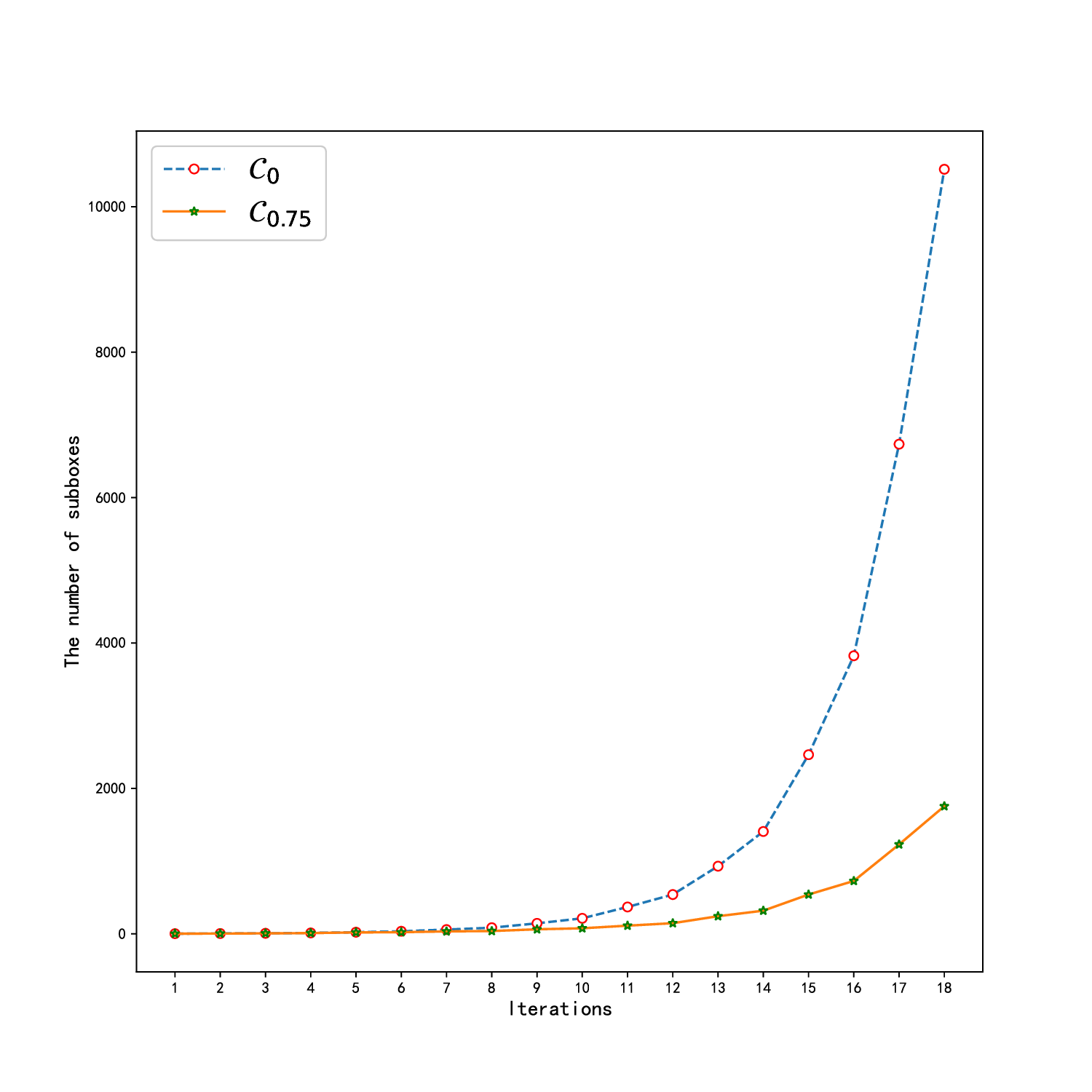}}
\caption{Results of CBB on scaled problems.}\label{fig3}
\end{figure}

\subsection{Test problems PEs}
We then attempt to solve 3-objective PE problems mentioned in \cite{ref64} by CBB. The PE1 is given as follows:
$$F(x)=\begin{pmatrix}
f_1(x)\\
f_2(x)\\
f_3(x)
\end{pmatrix}=
\begin{pmatrix}
\sum_{i=1}^n (x_i-a_i^{(1)})^2\\
\sum_{i=1}^n (x_i-a_i^{(2)})^2\\
\sum_{i=1}^n (x_i-a_i^{(3)})^2
\end{pmatrix},x\in[-2,2]^3,$$
where we set $a^{(1)}=(1,1,1)^T$, $a^{(2)}=(-1,-1,-1)^T$, and $a^{(3)}=(1,-1,1)^T$. The PE2 is given as follows:
$$F(x)=\begin{pmatrix}
f_1(x)+\frac{|f_1(x)+f_2(x)-12|}{2\sqrt{6}}\|x+e_2\|^2\\
f_2(x)+\frac{|f_1(x)+f_2(x)-12|}{2\sqrt{6}}\|x+e_2\|^2\\
f_3(x)
\end{pmatrix},x\in[-2,2]^3,$$
where $e_2=(0,1,0)^T$, and $f_j(x)$, $j=1,2,3$ are defined in PE1. Fig. \ref{fig4}(a) and (b) show the results of CBB with $\mathcal{C}_{(w,\theta_1)}$.

Next, we set $a^{(1)}=(-1,1,1)^T$, $a^{(2)}=(1,-1,1)^T$, and $a^{(3)}=(1,1,-1)^T$ in PE1, and we call this new problem PE3. To investigate the effect of cones on the distribution of the objective vectors, we use the polyhedral cone $\mathcal{C}_{0.75}$ and two ice cream cones $\mathcal{C}_{(w,\theta_1)}$ and $\mathcal{C}_{(w,\theta_2)}$. As mentioned above, this setup makes the three cones satisfy $\mathcal{C}_{(w,\theta_2)}\subseteq\mathcal{C}_{0.75}\subseteq\mathcal{C}_{(w,\theta_1)}$. Fig. \ref{fig4}(c) shows the results of CBB with three cones on PE3. Objective vectors with $\mathcal{C}_{(w,\theta_2)}$ (red dots) are shown on the Pareto front. Objective vectors with $\mathcal{C}_{(w,\theta_1)}$ (blue stars) and to $\mathcal{C}_{0.75}$ (green crosses) are shown with an offset to the Pareto front. It is easy to see that the larger the cone, the smaller the distribution of solutions.

To investigate the effect of different directions of the ice cream cone on the distribution of solutions, we use three different directions: $w_1=(0.1,0.5,0.5)^T$, $w_2=(0.5,0.1,0.5)^T$ and $w_3=(0.5,0.5,0.1)^T$. Fig. \ref{fig4}(d) shows the influence of the directions on the distribution of solutions. As expected, the algorithm correctly finds the solutions corresponding to the directions.
\begin{figure}[htbp]
\centering
\subfigure[PE1]{
\includegraphics[width=0.45\textwidth]{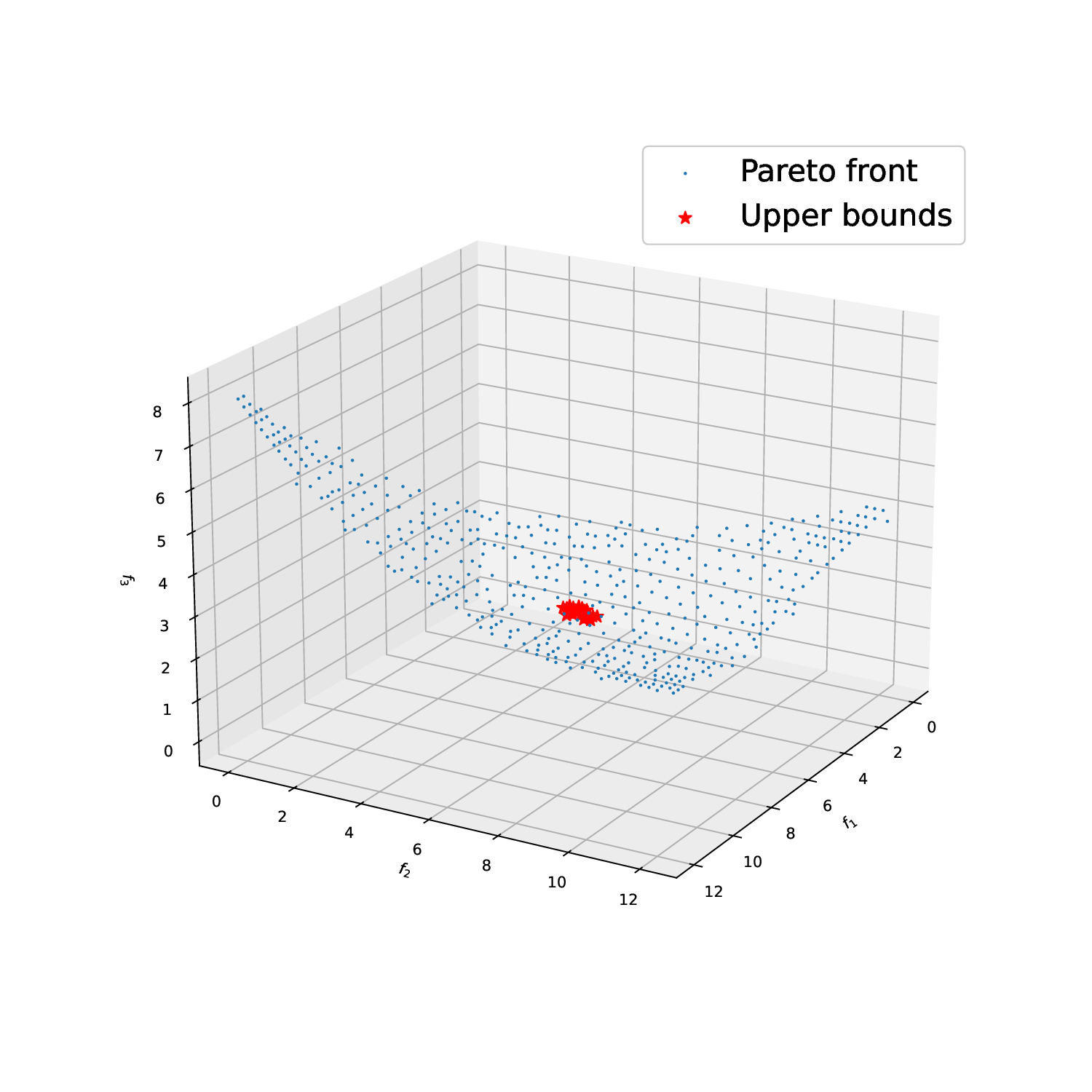}}
\subfigure[PE2]{
\includegraphics[width=0.45\textwidth]{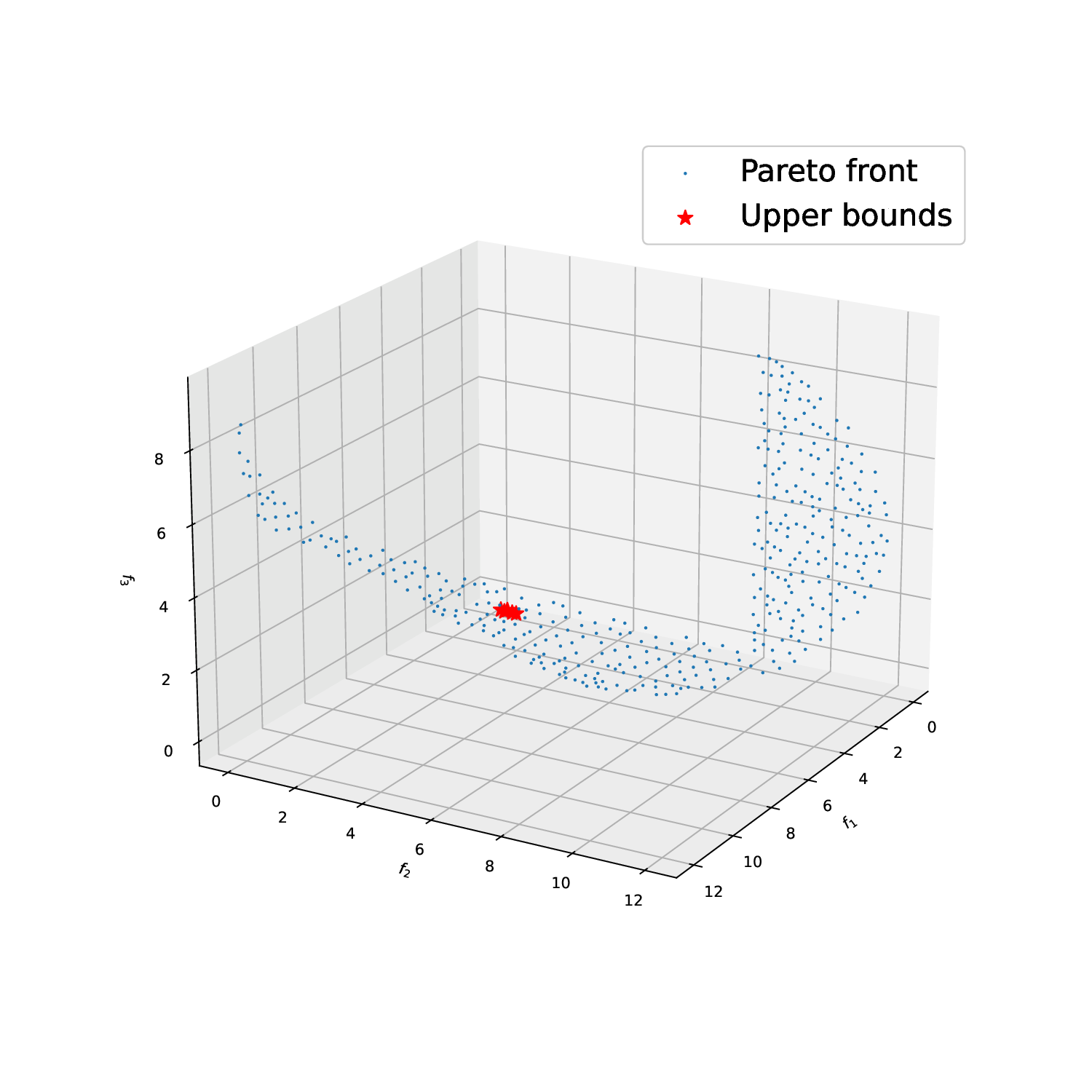}}
\subfigure[different cones]{
\includegraphics[width=0.45\textwidth]{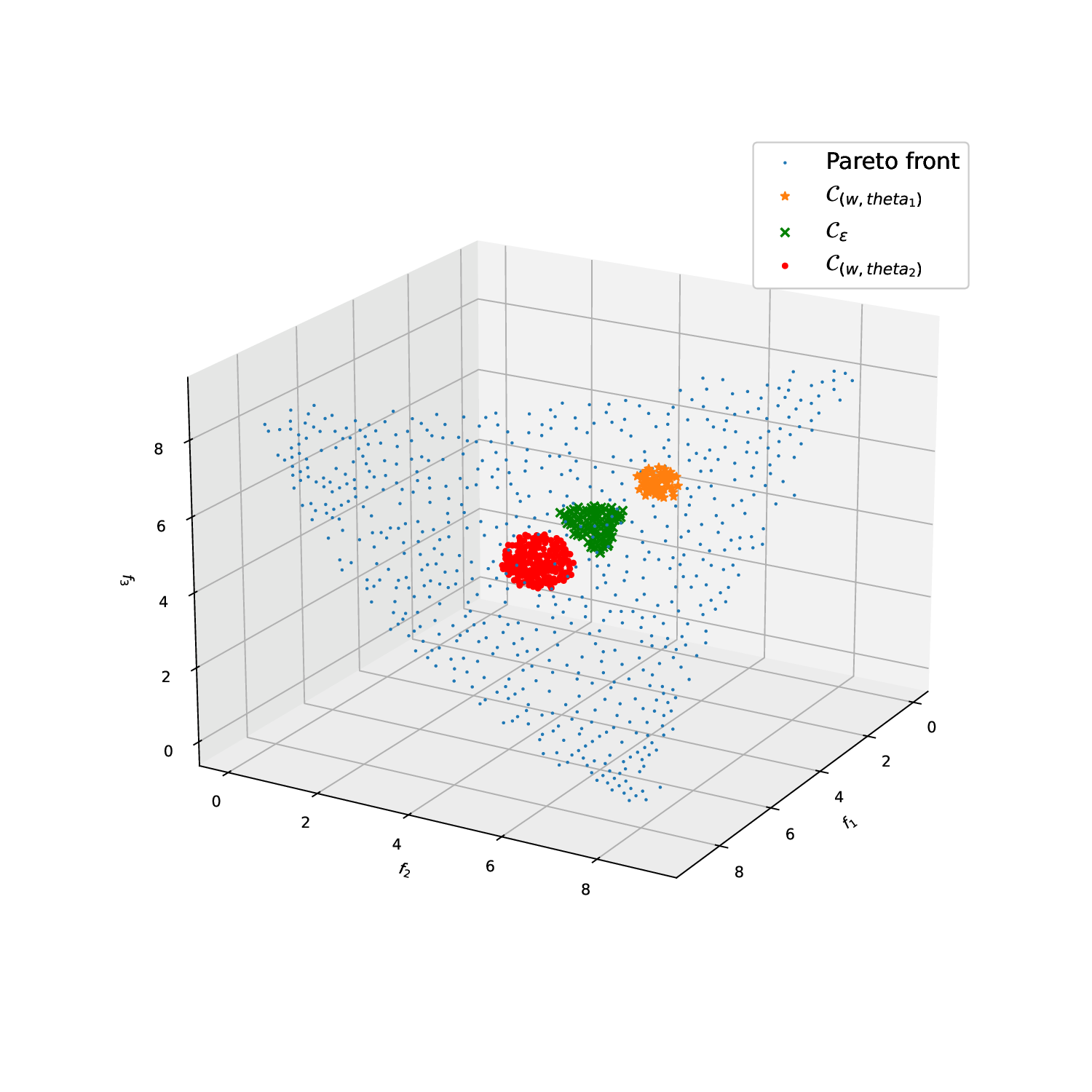}}
\subfigure[different directions]{
\includegraphics[width=0.45\textwidth]{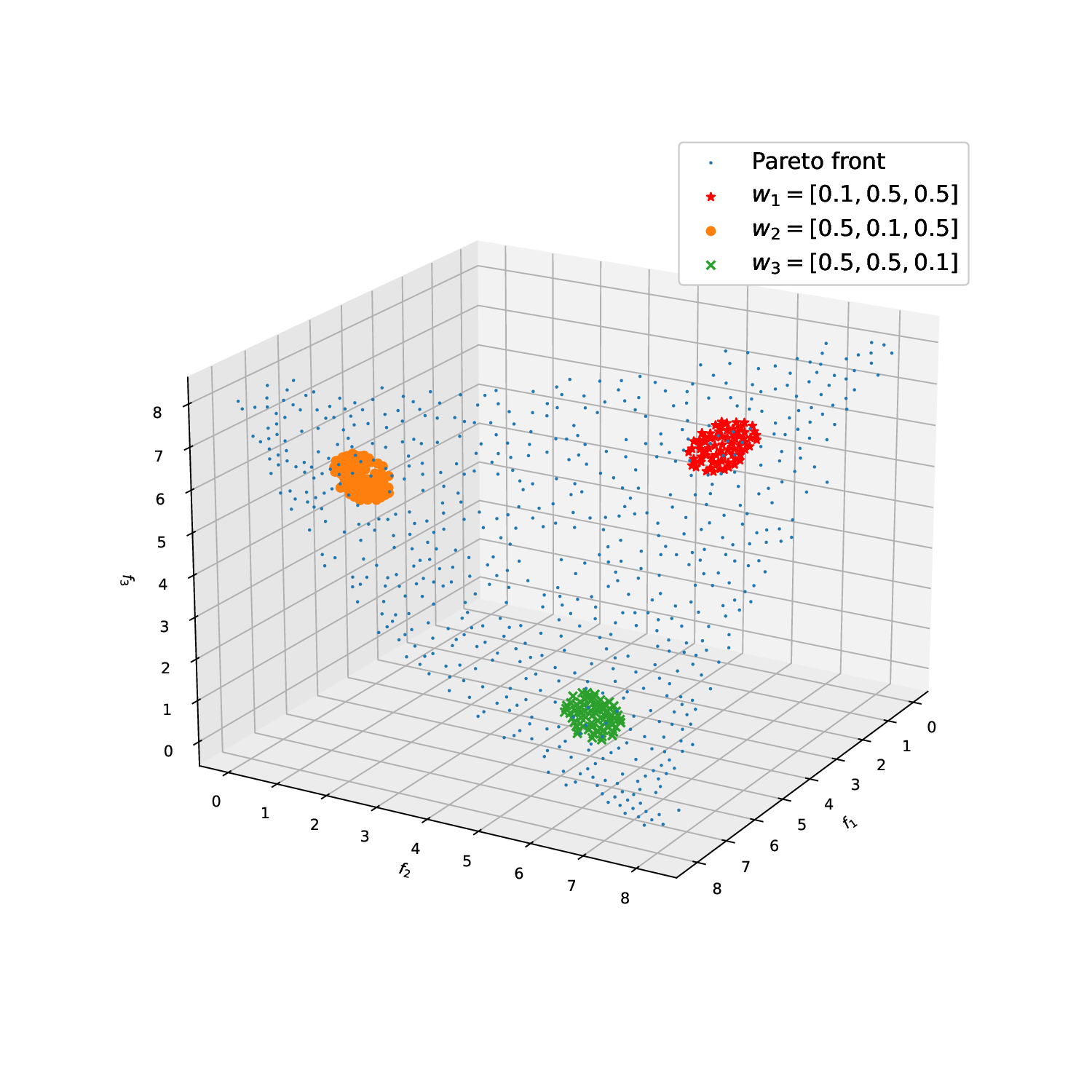}}
\caption{Results of CBB on PEs.}\label{fig4}
\end{figure}

\subsection{Constrained test problems}
We now consider three classical constrained test problems: SRN \cite{ref65}, CONSTR \cite{ref65} and KITA \cite{ref66}. We use CBB with $\mathcal{C}_{0.75}$ to solve these problems. Fig. \ref{fig5} provides the
results. It can be clearly seen that the inequality constraints do not cause any difficulty for the proposed algorithm.

\begin{figure}[htbp]
\centering
\subfigure[SRN]{
\includegraphics[width=0.30\textwidth]{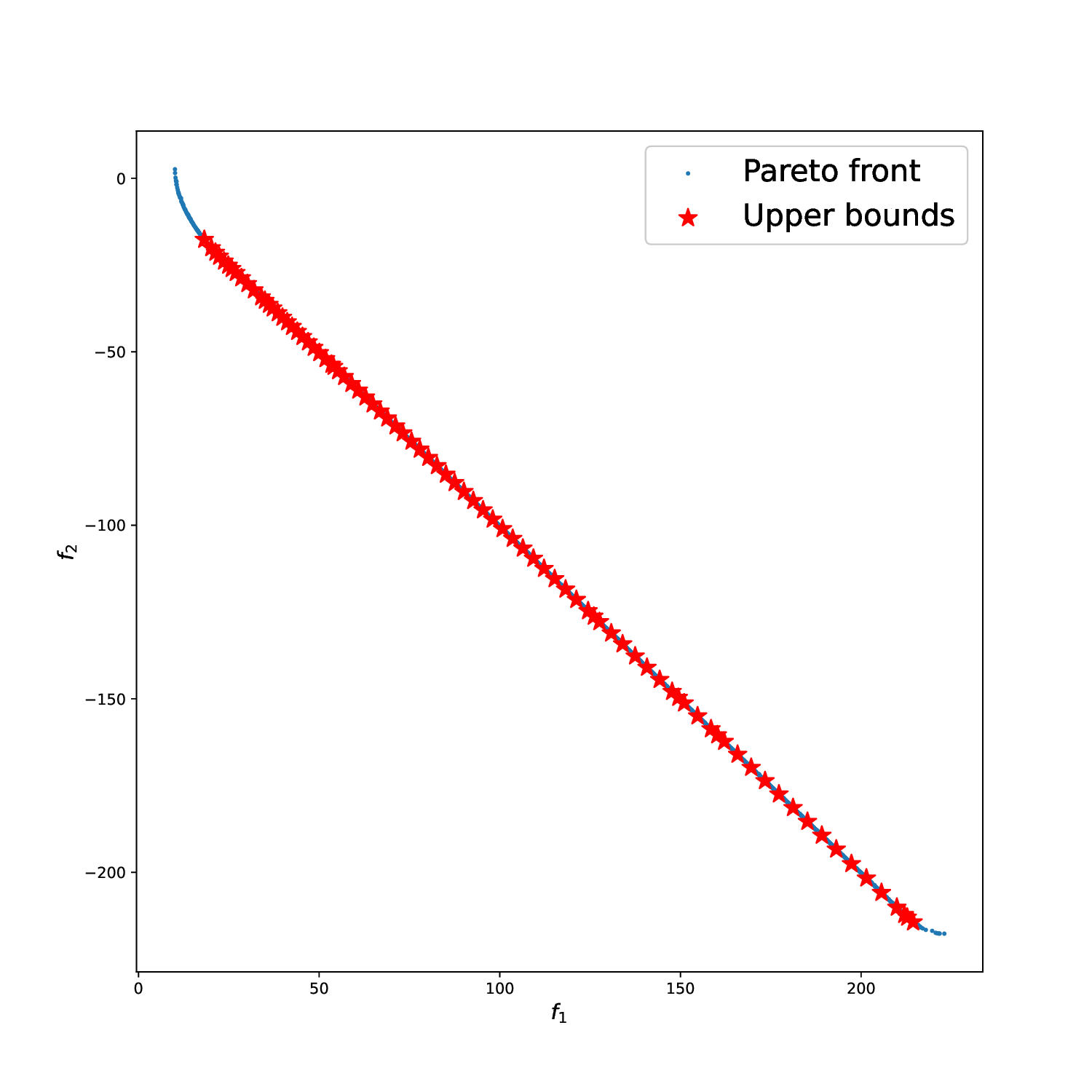}}
\subfigure[CONSTR]{
\includegraphics[width=0.30\textwidth]{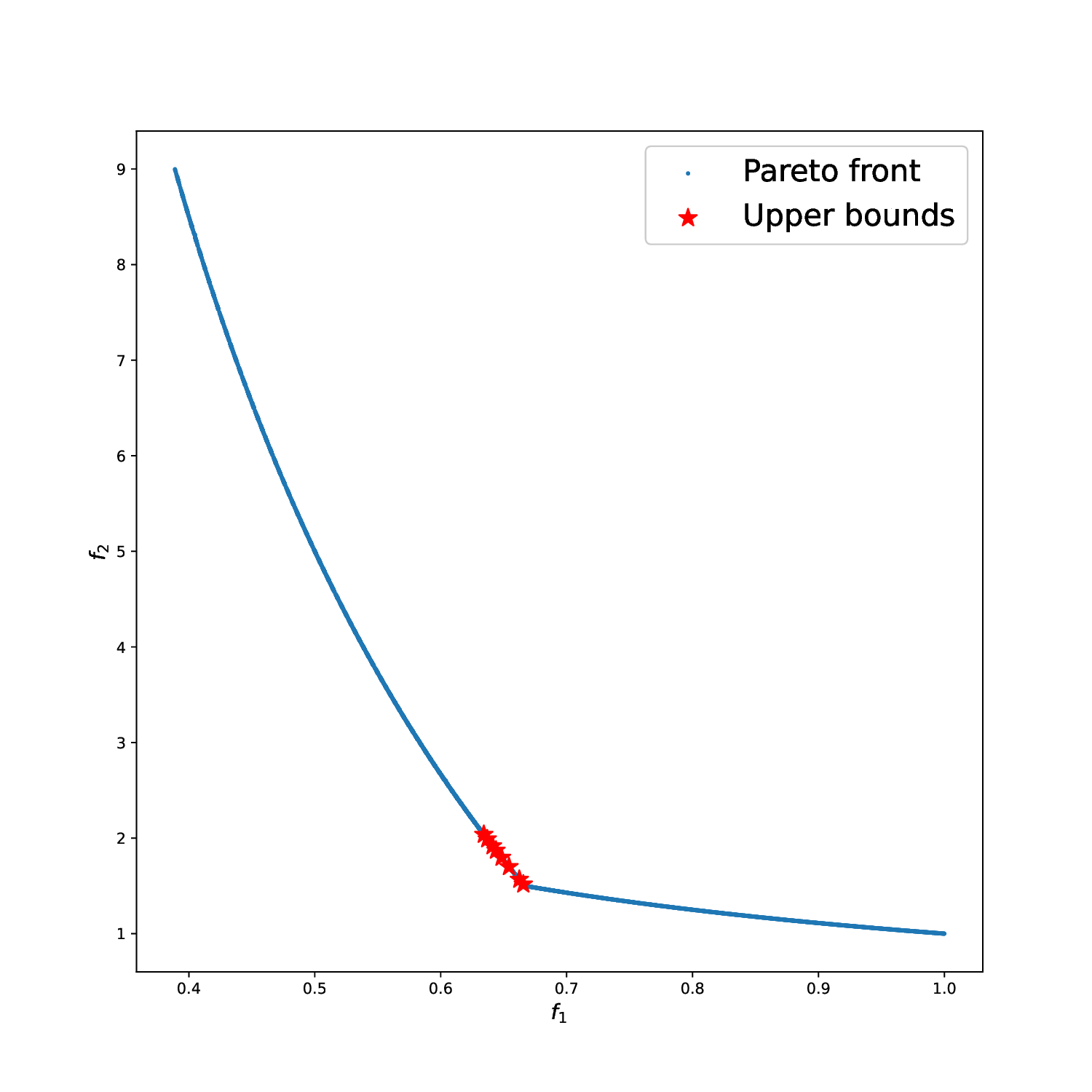}}
\subfigure[KITA]{
\includegraphics[width=0.30\textwidth]{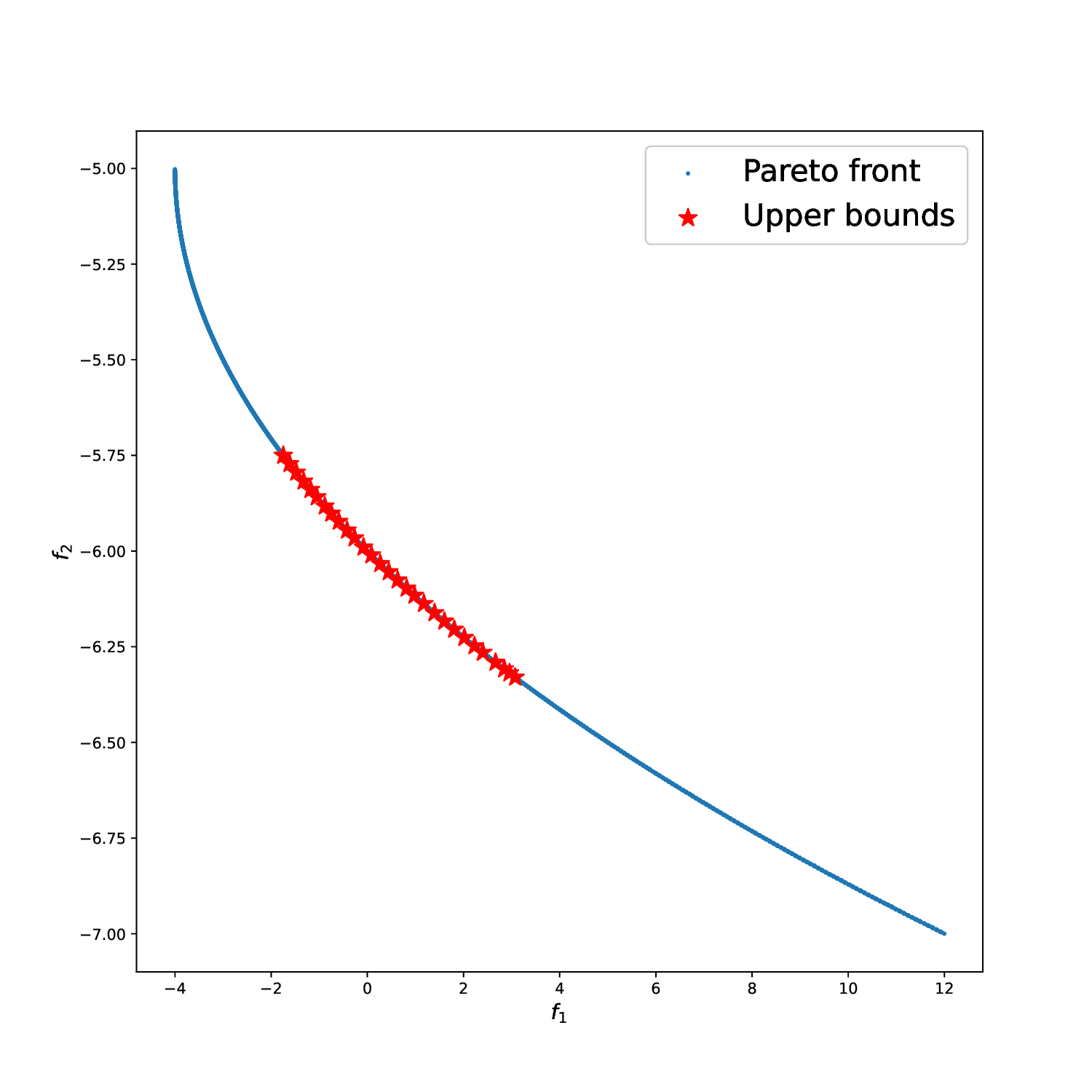}}
\caption{Results of CBB on three constrained problems.}\label{fig5}
\end{figure}

\subsection{Real-world constrained multiobjective optimization problems}
Finally, we apply CBB to several real-world constrained problems, including cantilever beam design, simply supported I-beam design, welded beam design, car side impact design problems and water resource management problems. These problems have 2- to 5-objectives and multiple inequality constraints. Specific descriptions of the problems can be found in the real-world constrained multiobjective optimization test-suite \cite{ref67}. For 2-objective problems, we use CBB with the cone $\mathcal{C}_{0.75}$; for 3 or more objective problems, we use CBB with the cone $\mathcal{C}_{(w,\theta_1)}$. All Pareto fronts are found by CBB with $\mathcal{C}_{0}$. The results are shown in Fig. \ref{fig6}. Obviously, CBB is applicable to two or more objectives, and linear or nonlinear constraints.

\begin{figure}[htbp]
\centering
%\subfigure[SRN]{
%\includegraphics[width=0.45\textwidth]{fig6a.eps}}
\subfigure[Cantilever beam design]{
\includegraphics[width=0.30\textwidth]{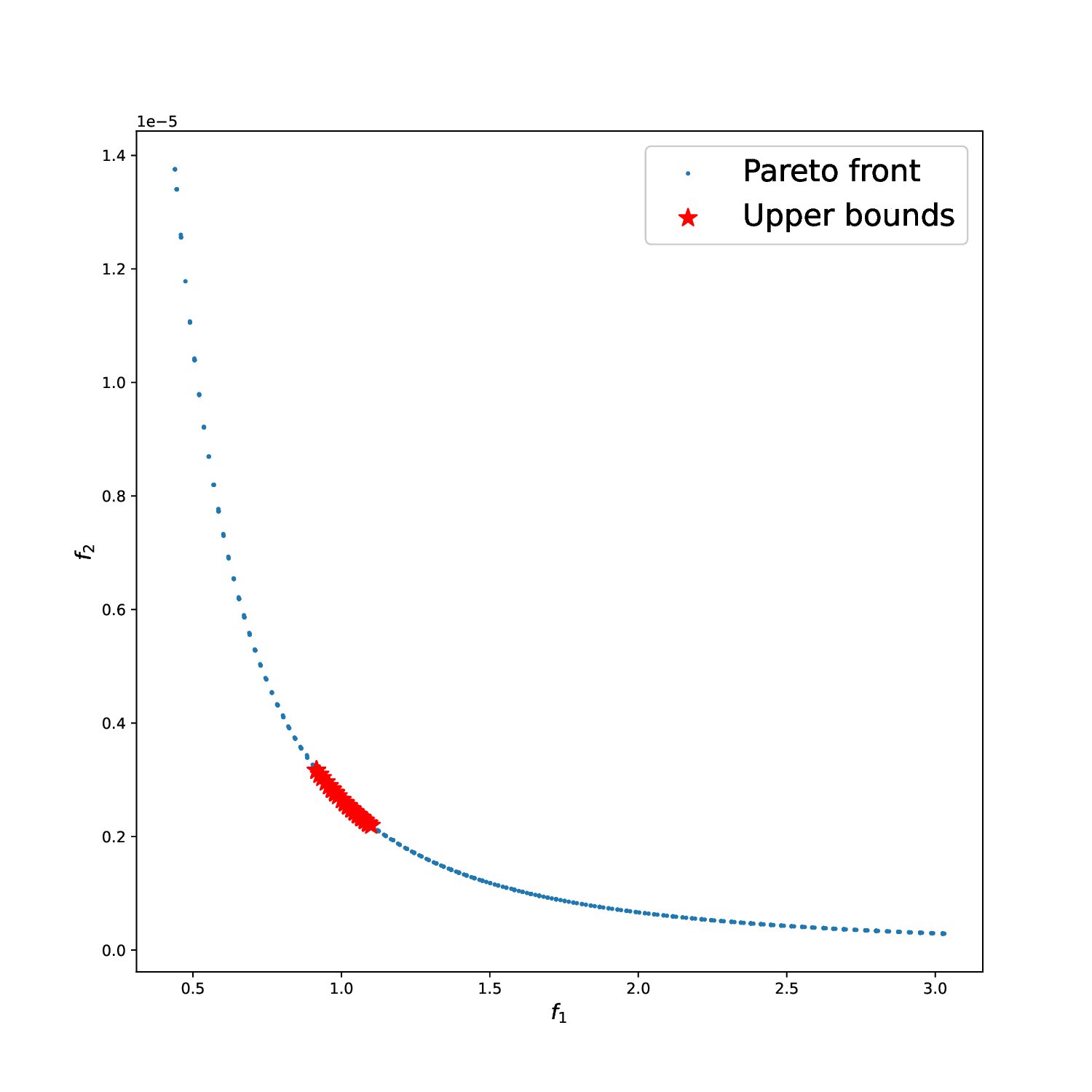}}
\subfigure[Simply supported I-beam design]{
\includegraphics[width=0.30\textwidth]{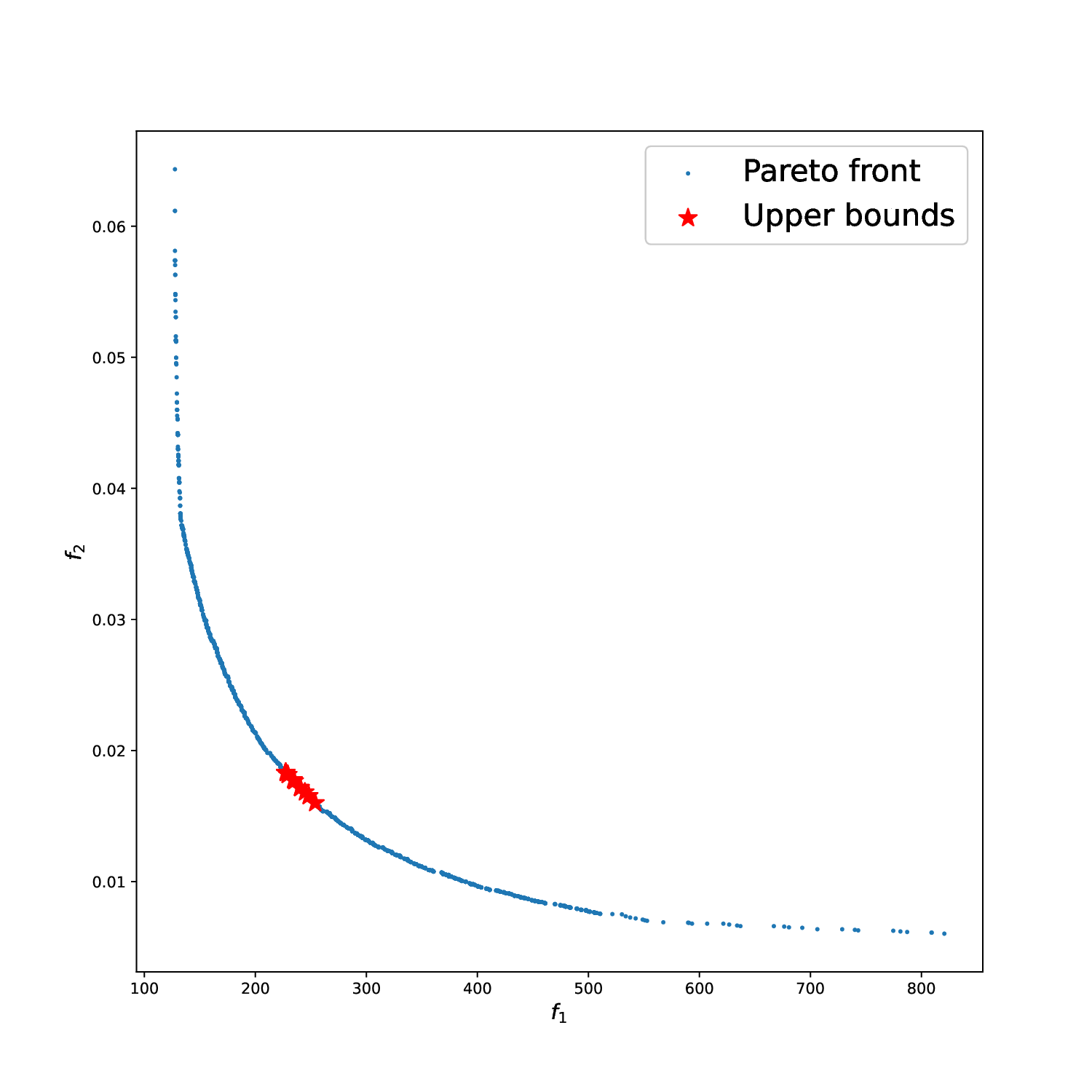}}
\subfigure[Welded beam design]{
\includegraphics[width=0.30\textwidth]{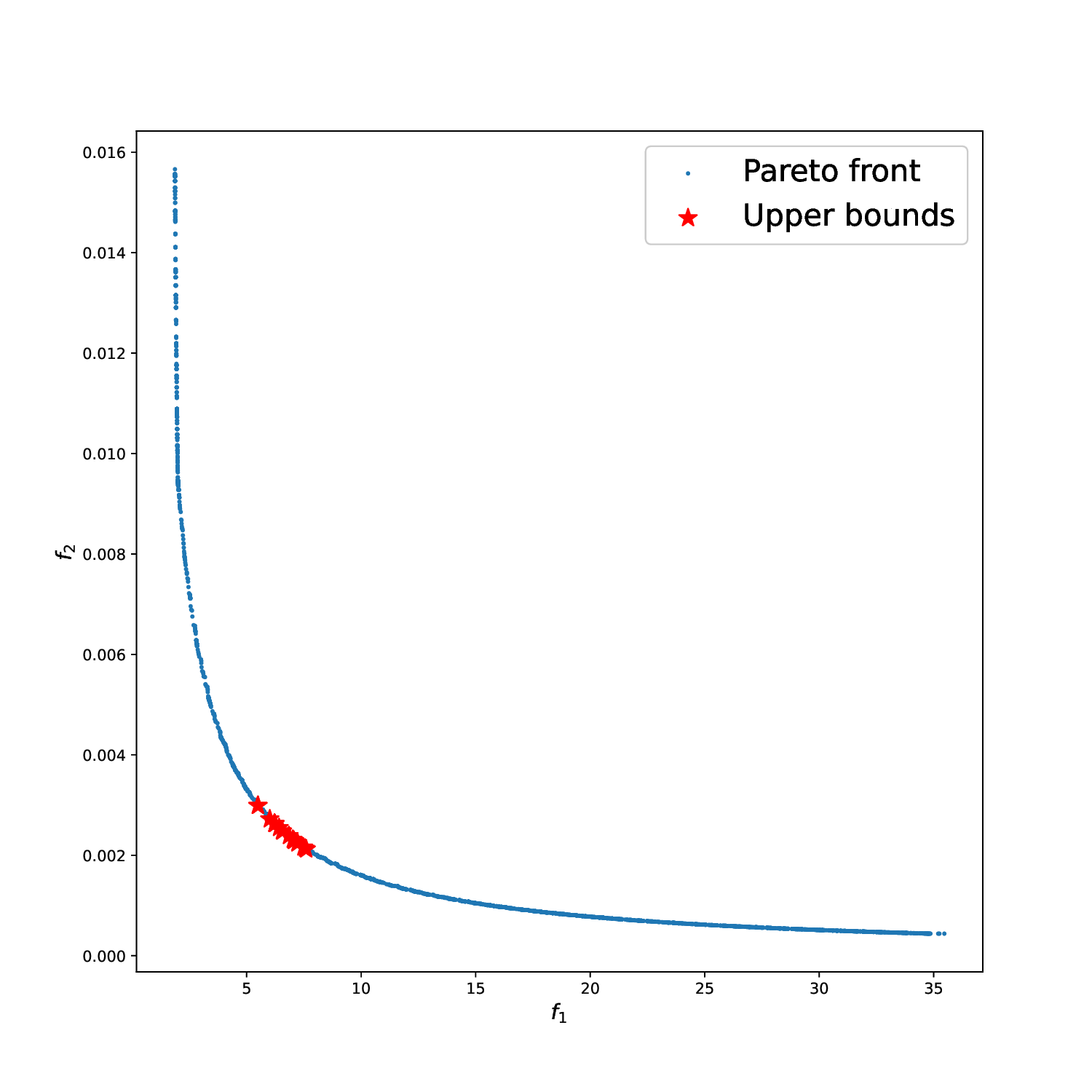}}
\subfigure[Car side impact design]{
\includegraphics[width=0.7\textwidth]{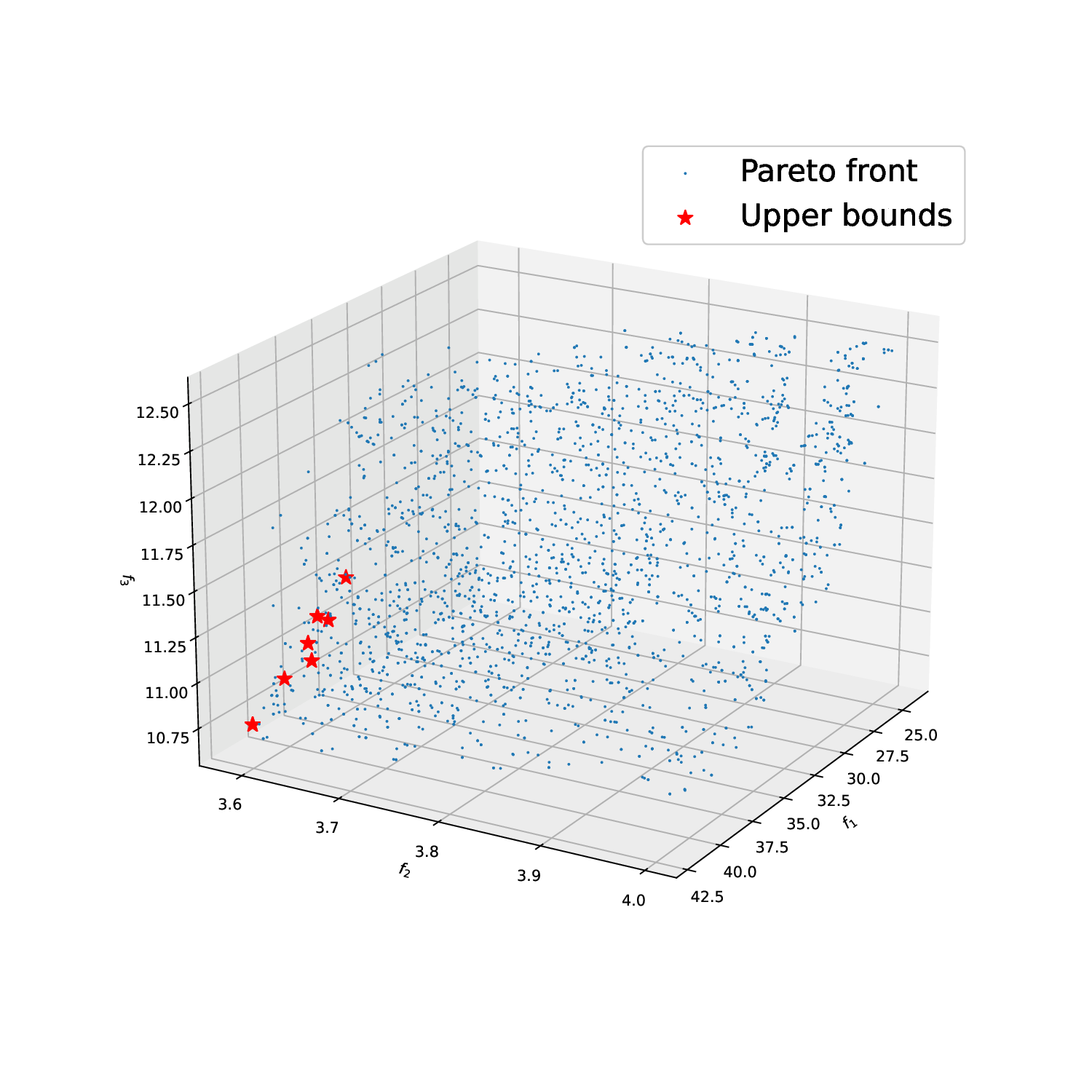}}
\subfigure[Water resource management]{
\includegraphics[width=0.9\textwidth]{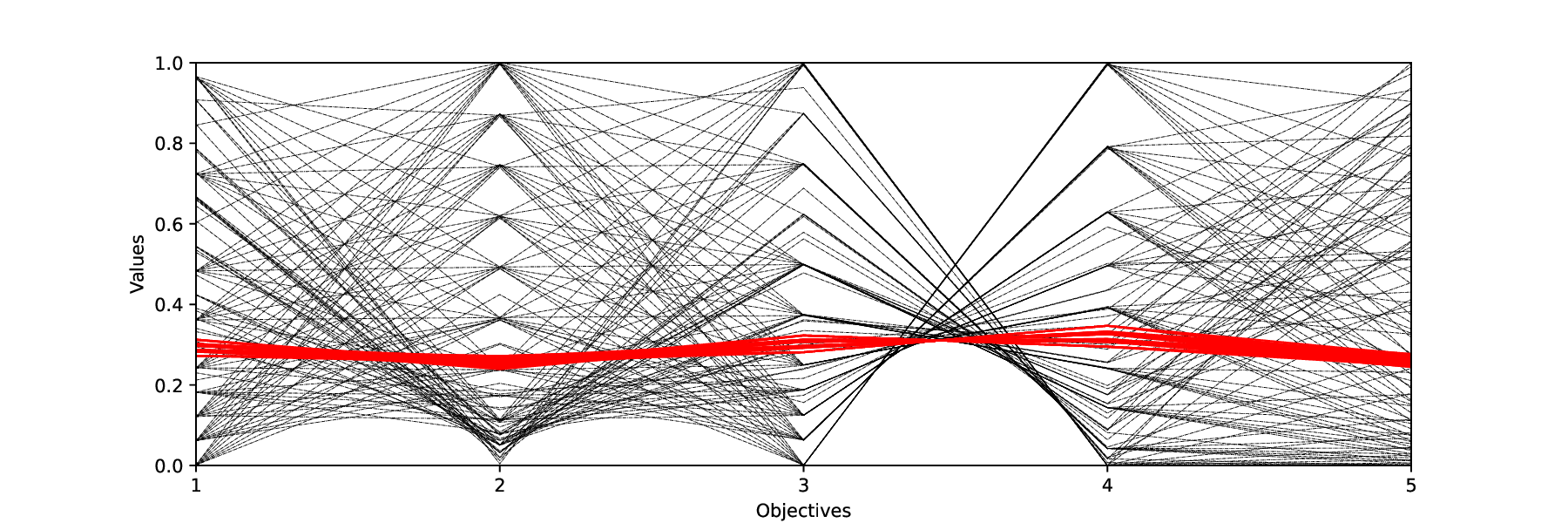}}
\caption{Results of CBB on real-world constrained problems.}\label{fig6}
\end{figure}

\section{Conclusion}

Many branch and bound algorithms for MOPs aim to approximate the entire Pareto optimal solution set. However, their solution processes are considered to be resource-intensive and time-consuming. An effective way to reduce their computational cost is to add additional preference information to the solution process. In this paper, we introduce preferences represented by ordering cones into the branch and bound algorithm, and propose the cone dominance-based branch and bound algorithm. The basic idea was to adopt the cone dominance induced by a pointed closed convex cone which is larger than $\mathbb{R}^m_+$ in the discarding test. In particular, we consider both polyhedral and non-polyhedral cones, and propose the corresponding discarding tests. This approach resulted in the removal of subboxes that did not contain efficient solutions with respect to the general cone, thereby significantly reducing the number of subboxes and candidate solutions. The efficacy and applicability of the proposed algorithm have been demonstrated through numerical experiments.

It should be noted that the algorithm is not proposed for vector optimization problems. However, from the description of the problem, there is no difference between problem \eqref{VP} and a general vector optimization problem. This is because our study does not apply to all pointed closed convex cones. Instead, our results are derived under the assumption that the considered cones contain $\mathbb{R}^m_+$. This assumption ensures that the transitivity of the cone orderings can be applied to Inequality \eqref{IE:2.2} so as to justify the cone dominance-based discarding tests. If a lower bound which cone-dominates any solution in a subbox can be obtained, then the cone dominance-based discarding test remains valid even in the absence of this assumption. Consequently, the proposed algorithm can be directly applied to the vector optimization case.

\begin{acknowledgements}
This work is supported by the Major Program of National Natural Science Foundation of China (Nos. 11991020, 11991024), the General Program of National Natural Science Foundation of China (No. 11971084), the Team Project of Innovation Leading Talent in Chongqing (No. CQYC20210309536), the NSFC-RGC (Hong Kong) Joint Research Program (No. 12261160365), and the Scientific Research Foundation for the PhD (Ningbo University of Technology, No. 2090011540025).
\end{acknowledgements}


\begin{thebibliography}{99}


\bibitem{ref62} Aliprantis, C. D., Monteiro, P. K., \& Tourky, R. (2004). Non-marketed options, non-existence of equilibria, and non-linear prices. Journal of Economic Theory, 114(2), 345-357.

\bibitem{ref3} Branke, J., Deb, K., Dierolf, H., \& Osswald, M. (2004). Finding knees in multi-objective optimization. In Parallel Problem Solving from Nature-PPSN VIII: 8th International Conference, Birmingham, UK, September 18-22, 2004. Proceedings 8 (pp. 722-731). Springer Berlin Heidelberg.

\bibitem{ref56} Cacchiani, V., \& D'Ambrosio, C. (2017). A branch-and-bound based heuristic algorithm for convex multi-objective MINLPs. European Journal of Operational Research, 260(3), 920-933.

\bibitem{ref64} Cuate, O., \& Sch\"{u}tze, O. (2020). Pareto explorer for finding the knee for many objective optimization problems. Mathematics, 8(10), 1651.

\bibitem{ref57} De Santis, M., Eichfelder, G., Niebling, J., \& Rockt\"{a}schel, S. (2020). Solving multiobjective mixed integer convex optimization problems. SIAM Journal on Optimization, 30(4), 3122-3145.

\bibitem{ref65} Deb, K., Pratap, A., Agarwal, S., \& Meyarivan, T. A. M. T. (2002). A fast and elitist multiobjective genetic algorithm: NSGA-II. IEEE transactions on evolutionary computation, 6(2), 182-197.

\bibitem{ref10} Eichfelder, G., Kirst, P., Meng, L., \& Stein, O. (2021). A general branch-and-bound framework for continuous global multiobjective optimization. Journal of Global Optimization, 80, 195-227.

\bibitem{ref60} Eichfelder, G., \& Stein, O. (2024). Limit sets in global multiobjective optimization. Optimization, 73(1), 1-27.

\bibitem{ref63} Engau, A., \& Wiecek, M. M. (2007). Cone characterizations of approximate solutions in real vector optimization. Journal of Optimization Theory and Applications, 134(3), 499-513.

\bibitem{ref11} Fern{\'a}ndez, J., \& T{\'o}th, B. (2009). Obtaining the efficient set of nonlinear biobjective optimization problems via interval branch-and-bound methods. Computational Optimization and Applications, 42, 393-419.
\bibitem{ref69} Forget, N., Gadegaard, S. L., \& Nielsen, L. R. (2022). Warm-starting lower bound set computations for branch-and-bound algorithms for multi objective integer linear programs. European Journal of Operational Research, 302(3), 909-924.

\bibitem{ref22} Ikeda, K., Kita, H., \& Kobayashi, S. (2001). Failure of Pareto-based MOEAs: Does non-dominated really mean near to optimal?. In Proceedings of the 2001 congress on evolutionary computation (IEEE Cat. No. 01TH8546) (Vol. 2, pp. 957-962). IEEE.

\bibitem{ref15} Jain, H., \& Deb, K. (2013). An evolutionary many-objective optimization algorithm using reference-point based nondominated sorting approach, part II: Handling constraints and extending to an adaptive approach. IEEE Transactions on evolutionary computation, 18(4), 602-622.

\bibitem{ref46} Jan, M. A., \& Zhang, Q. (2010, September). MOEA/D for constrained multiobjective optimization: Some preliminary experimental results. In 2010 UK Workshop on computational intelligence (UKCI) (pp. 1-6). IEEE.


\bibitem{ref66} Kita, H., Yabumoto, Y., Mori, N., \& Nishikawa, Y. (1996). Multi-objective optimization by means of the thermodynamical genetic algorithm. In Parallel Problem Solving from Nature¡ªPPSN IV: International Conference on Evolutionary Computation¡ªThe 4th International Conference on Parallel Problem Solving from Nature Berlin, Germany, September 22¨C26, 1996 Proceedings 4 (pp. 504-512). Springer Berlin Heidelberg.
\bibitem{ref68}  Klamroth, K., Lacour, R., \& Vanderpooten, D. (2015). On the representation of the search region in multi-objective optimization. European Journal of Operational Research, 245(3), 767-778.

\bibitem{ref67} Kumar, A., Wu, G., Ali, M. Z., Luo, Q., Mallipeddi, R., Suganthan, P. N., \& Das, S. (2021). A benchmark-suite of real-world constrained multi-objective optimization problems and some baseline results. Swarm and Evolutionary Computation, 67, 100961.

\bibitem{ref23} Miettinen, K. (1999). Nonlinear multiobjective optimization (Vol. 12). Springer Science \& Business Media.

\bibitem{ref25} Neumaier, A. (1990). Interval methods for systems of equations (No. 37). Cambridge university press.

\bibitem{ref26} Niebling, J., \& Eichfelder, G. (2019). A branch--and--bound-based algorithm for nonconvex multiobjective optimization. SIAM Journal on Optimization, 29(1), 794-821.


\bibitem{ref55} Przybylski, A., \& Gandibleux, X. (2017). Multi-objective branch and bound. European Journal of Operational Research, 260(3), 856-872.


\bibitem{ref58} Sawaragi, Y., Nakayama, H., \& Tanino, T. (Eds.). (1985). Theory of multiobjective optimization. Elsevier.

\bibitem{ref37} Sch\"{u}tze, O., Laumanns, M., \& Coello, C. A. C. (2008). Approximating the knee of an MOP with stochastic search algorithms. In Parallel Problem Solving from Nature¨CPPSN X: 10th International Conference, Dortmund, Germany, September 13-17, 2008. Proceedings 10 (pp. 795-804). Springer Berlin Heidelberg.

\bibitem{ref61} Tanino, T., Tanaka, T., Inuiguchi, M., Hunt, B. J., \& Wiecek, M. M. (2003). Cones to aid decision making in multicriteria programming. In Multi-Objective Programming and Goal Programming: Theory and Applications (pp. 153-158). Springer Berlin Heidelberg.

\bibitem{ref38} Wierzbicki, A. P. (1980). The use of reference objectives in multiobjective optimization. In Multiple Criteria Decision Making Theory and Application: Proceedings of the Third Conference Hagen/K\"{o}nigswinter, West Germany, August 20¨C24, 1979 (pp. 468-486). Springer Berlin Heidelberg.

\bibitem{ref59} Wierzbicki, A. P. (1986). A methodological approach to comparing parametric characterizations of efficient solutions. In Large-Scale Modelling and Interactive Decision Analysis: Proceedings of a Workshop sponsored by IIASA (International Institute for Applied Systems Analysis) and the Institute for Informatics of the Academy of Sciences of the GDR Held at the Wartburg Castle, Eisenach, GDR, November 18¨C21, 1985 (pp. 27-45). Springer Berlin Heidelberg.


\bibitem{ref44} Wu, W. T., \& Yang, X. M. (2023). Reference-point-based branch and bound algorithm for multiobjective optimization. Journal of Global Optimization, 1-19.

\bibitem{ref20} Zhang, Q., \& Li, H. (2007). MOEA/D: A multiobjective evolutionary algorithm based on decomposition. IEEE Transactions on evolutionary computation, 11(6), 712-731.

\bibitem{ref42} \v{Z}ilinskas, A., \& Gimbutien\.{e}, G. (2016). On one-step worst-case optimal trisection in univariate bi-objective Lipschitz optimization. Communications in Nonlinear Science and Numerical Simulation, 35, 123-136.

\bibitem{ref43} \v{Z}ilinskas, A., \& \v{Z}ilinskas, J. (2015). Adaptation of a one-step worst-case optimal univariate algorithm of bi-objective Lipschitz optimization to multidimensional problems. Communications in Nonlinear Science and Numerical Simulation, 21(1-3), 89-98.


\end{thebibliography}
\end{document}